\documentclass[a4paper,11pt]{article}
\usepackage[utf8]{inputenc}
\usepackage{amsmath}
\usepackage{amsfonts}
\usepackage{amssymb}
\usepackage{amsthm}
\usepackage[english]{babel}
\usepackage{fontenc}
\usepackage{graphicx}
\usepackage{subfigure}
\usepackage{mathtools}
\usepackage{accents}
\mathtoolsset{showonlyrefs}
\usepackage{dsfont}
\usepackage[hmargin=2.5cm,vmargin=2.5cm,paper=a4paper]{geometry}
\newcommand{\ls}{\leqslant}
\newcommand{\gs}{\geqslant}
\newcommand{\wksto}{\stackrel{\ast}{\rightharpoonup}}
\renewcommand{\div}{\operatorname{div}}
\newcommand{\R}{\mathbb R}
\newcommand{\per}{\operatorname{Per}}
\newcommand{\sign}{\operatorname{sign}}
\newcommand{\supp}{\operatorname{supp}}

\newcommand{\eps}{\varepsilon}

\newcommand{\loc}{\mathrm{loc}}
\newcommand{\BV}{\mathrm{BV}}
\newcommand{\BD}{\mathrm{BD}}
\newcommand{\dd}{\, \mathrm{d}}
\newcommand{\TV}{\operatorname{TV}}
\newcommand{\TD}{\operatorname{TD}}
\newcommand{\TGV}{\operatorname{TGV}}

\newcommand{\dom}{\operatorname{dom}}

\newcommand{\scal}[2]{ \left \langle #1, \, #2 \right \rangle}

\newcommand\restr[2]{{
\left.\kern-\nulldelimiterspace #1 \vphantom{\big|} \right|_{#2} 
}}
\newcommand{\mres}{\mathbin{\vrule height 1.6ex depth 0pt width
0.13ex\vrule height 0.13ex depth 0pt width 1.3ex}}

\DeclareMathOperator*{\argmin}{arg\,min}

\newtheorem{theorem}{Theorem}
\newtheorem{prop}{Proposition}
\newtheorem{cor}{Corollary}

\newtheorem{lemma}{Lemma}
\theoremstyle{definition}
\newtheorem{definition}{Definition}
\newtheorem{remark}{Remark}
\newtheorem{example}{Example}
\usepackage{color}
\usepackage[dvipsnames]{xcolor}
\usepackage{comment}
\usepackage[normalem]{ulem}

\begin{document}
\title{Boundedness and unboundedness in total variation regularization\footnotetext{2020 Mathematics Subject Classification: 49Q20, 47A52, 65J20, 65J22.}\footnotetext{Keywords: Total variation, linear inverse problems, boundedness of minimimizers, generalized taut string, vanishing weights, infimal convolution regularizers}}
\author{Kristian Bredies\thanks{\raggedright{Institute of Mathematics and Scientific Computing, University of Graz, Austria (\texttt{kristian.bredies@uni-graz.at})}}, Jos\'e A. Iglesias\thanks{Department of Applied Mathematics, University of Twente, The Netherlands (\texttt{jose.iglesias{@}utwente.nl})} \ and Gwenael Mercier\thanks{Faculty of Mathematics, University of Vienna, Austria (\texttt{gwenael.mercier{@}univie.ac.at})}}
\date{\vspace{-1.7em}}
\maketitle

\begin{abstract}
We consider whether minimizers for total variation regularization of linear inverse problems belong to $L^\infty$ even if the measured data does not. We present a simple proof of boundedness of the minimizer for fixed regularization parameter, and derive the existence of uniform bounds for sufficiently small noise under a source condition and adequate a priori parameter choices. To show that such a result cannot be expected for every fidelity term and dimension we compute an explicit radial unbounded minimizer, which is accomplished by proving the equivalence of weighted one-dimensional denoising with a generalized taut string problem. Finally, we discuss the possibility of extending such results to related higher-order regularization functionals, obtaining a positive answer for the infimal convolution of first and second order total variation.
\end{abstract}
 
\section{Introduction}\label{sec:intro}
For $\Omega \subset \R^d$ either a bounded Lipschitz domain or the whole $\R^d$ with $d \geq 2$, $\Sigma \subset \R^m$ an arbitrary domain with $m \gs 1$, and given a linear bounded operator
\[A: L^{d/(d-1)}(\Omega) \to L^q(\Sigma),\]
we are interested in solutions $u_{\alpha, w}$ of the $\TV$-regularized inverse problem $Au=f$ with noisy data $f+w$ for $w \in L^q(\Sigma)$, that is, solutions of the minimization problem
\begin{equation}\label{eq:primalproblem}\min_{ u \in L^{d/(d-1)}(\Omega)} \frac{1}{\sigma}\left( \int_\Sigma |Au - (f+w)|^q \right)^{\sigma/q}+ \alpha \TV(u),\end{equation}
where $1< q < \infty$, $\sigma = \min(q,2)$ and $\TV(u)$ denotes the total variation of $u$, see \eqref{eq:TVdef} below for its definition. This convex minimization leads, at any minimizer $u_{\alpha, w}$, to the optimality condition
\begin{equation}\label{eq:optimality}\begin{aligned}
v_{\alpha,w} &:= A^\ast p_{\alpha, w} \in \partial_{L^{d/(d-1)}} \TV(u_{\alpha,w}) \subset L^d(\Omega), \\ p_{\alpha, w} &:= \frac{1}{\alpha} \|Au_{\alpha,w}-f-w\|_{L^q(\Sigma)}^{\sigma-2} \, j(f+w-Au_{\alpha,w}),
\end{aligned}\end{equation}
where $j(u)=|u|^{q-2}u$ is the duality mapping of $L^q(\Sigma)$  and $\partial_{L^{d/(d-1)}} \TV( u_\alpha)$ denotes the subgradient (see Definition \ref{def:subgrad} below). Our main goal is to present a proof of the following uniform boundedness result:
\begin{theorem}\label{thm:joint_thm}
Assume that for $A,f$ there is a unique solution $u^\dag$ for $Au=f$ which satisfies the source condition $\operatorname{Ran}(A^\ast) \cap \partial \TV(u^\dag) \neq \emptyset$. There is some constant $C(q,\sigma, d, \Omega)$ such that if $\alpha_n, w_n$ are sequences of regularization parameters and perturbations for which
\[\alpha_n \gs C(q,\sigma, d, \Omega) \|A^\ast\| \Vert w_n \Vert_{L^q(\Sigma)}^{\sigma-1} \xrightarrow[n \to \infty]{} 0,\]
then the corresponding sequence $u_{\alpha_n, w_n}$ of minimizers is bounded in $L^\infty(\Omega)$, and (possibly up to a subsequence)
\[u_{\alpha_n, w_n} \xrightarrow[n \to \infty]{} u^\dag\text{ strongly in }L^{\overline p}(\Omega)\text{ for all }\overline p \in (1, \infty).\]
\end{theorem}
This result is a combination of Proposition \ref{prop:single_u_bounded}, Proposition \ref{prop:bounded_domains} and Corollary \ref{cor:strong_conv} in Section \ref{sec:boundedness}, which in turn depend on the preliminaries reviewed in Section \ref{sec:prelim}.

In the case of denoising $A=\operatorname{Id}$, it is well known that $\TV$ regularization satisfies a maximum principle and $\|u_{\alpha, w}\|_{L^\infty} \ls \|f+w\|_{L^\infty}$; a proof can be found in \cite[Lem.~2.1]{Cha04} for $L^2(\R^2)$ and \cite[Prop.~3.6]{IglMer21} for the $L^{d/(d-1)}(\R^d)$ case that most closely resembles the situation considered here. Another related result is the nonexpansiveness in $L^p$ norm for denoising with rectilinear anisotropy found in \cite{KirSet19}. In comparison, our results work with unbounded measurements and linear operators, and although the analysis depends on the noise level and parameter choice, the bound on $\|u_{\alpha, w}\|_{L^\infty}$ can be made uniform in the regime of small noise and regularization parameter. Our proof of Theorem \ref{thm:joint_thm} hinges on the $L^d$ summability and stability in $L^d$ norm of the subgradients $v_{\alpha,w}$ appearing in \eqref{eq:optimality} and their relation with the perimeter of the level sets of $u_{\alpha, w}$. To the best of our knowledge, the subgradient of $\TV$ was first characterized in \cite{AndBalCasMaz01a, AndBalCasMaz01b} (see also \cite{AndCasMaz04} for further context) relying on the results of \cite{Anz83}. In Section \ref{sec:prelim} we summarize the parts of these results that we will use in the sequel.

Moreover, we are also interested in finding cases outside the assumptions of the result above in which minimizers fail to be bounded. To this end, in Section \ref{sec:string} we consider a one-dimensional denoising problem \eqref{eq:weightedROF}, with weights present in both the fidelity term and the total variation which may vanish at the boundary of the interval. By proving its equivalence to a generalized taut string problem, one obtains that the minimizers can be constructed by gluing a few different types of behavior on finitely many subintervals. This in turn allows us to produce explicit unbounded minimizers for radial data in 3D and boundedness conditions for radial powers. Moreover, we believe that results on weighted taut-string formulations can be of independent interest. 

Indeed, the taut string approach to one-dimensional total variation minimization has been studied in many works, either from a continuous \cite{Gra07} or discrete \cite{Con13, HinEtAl03} point of view. Versions for more general variants have also been considered, like \cite{PoeSch08} for higher-order total variation, \cite{GraObe08} for `nonuniform tubes' which can be seen as a weight imposed directly on the taut string formulation, and \cite{KirSchSet19} where the graph setting is considered. Likewise, there is a large number of works considering weighted TV denoising seen as having spatially-dependent regularization parameters, and approaches to their automatic selection. This literature is extensive and the choice of weights for particular tasks is beyond the scope of this work, so we only explicitly mention the analytical studies \cite{HinPapRau17} for the one-dimensional case and \cite{AthJerNovOrl17} for higher dimensions and weights that degenerate at the boundary. In contrast, we have not been able to find material combining both weighted TV functionals and taut string characterizations, which motivates our investigations below.

In Section \ref{sec:higherorder} we explore the possibility of extending Theorem \ref{thm:joint_thm} or parts thereof to the setting of higher-order regularizers related to the total variation. We obtain a boundedness result for the case of infimal convolution of $\TV$ and $\TV^2$ regularizers, for which the optimality conditions are closely related to subgradients of $\TV$. Finally, we also present counterexamples suggesting that $L^\infty$ bounds in the case of total generalized variation ($\TGV$) likely need methods different to those considered in this work.

\section{Preliminaries}\label{sec:prelim}
In this section we collect definitions and preliminary results.
\begin{definition}\label{def:subgrad}For a convex functional $F:X\to \R \cup \{+\infty\}$, where $X$ is a Banach space, the subgradient or subdifferential of $F$ at some element $u \in X$ is the set defined as
\[\partial_X F(u) = \left\{ v \in X' \,\middle\vert\, F(\bar{u}) - F(u) \gs \langle v, \bar{u} - u\rangle_{(X',X)} \text{ for all } \bar{u} \in X\right\},\]
where $X'$ is the dual Banach space of $X$ and $\langle \cdot, \cdot \rangle_{(X',X)}$ denotes the corresponding duality product.
\end{definition}
\subsection{Subgradient of TV, pairings, and slicing}
The total variation is defined as
\begin{equation}\label{eq:TVdef}\TV(u;\Omega) := \sup\left\{ \int_\Omega u(x) \div z(x) \dd x \, \middle\vert\, z \in C^1_c(\Omega;\R^d) \text{ with } |z(x)| \ls 1 \text{ for all } x \in \Omega\right\},\end{equation}
which in turn motivates defining the perimeter in $\Omega$ of a Lebesgue measurable subset $E \subset \Omega$ as $\per(E):=\TV(1_E)$, where $1_E$ is the indicatrix taking the value $1$ on $E$ and $0$ on $\R^d \setminus E$.

We notice that \eqref{eq:TVdef} is the Fenchel conjugate with respect to the dual pair $(L^d, L^{d/(d-1)})$ of the convex set $K:=\div \{z \in C^1_c(\Omega;\R^d), \|z\|_{L^\infty(\Omega)} \ls 1\}$, so that 
\begin{equation}\label{eq:tvast}(\TV)^\ast=\big((\chi_K)^\ast\big)^\ast = \chi_{\overline K},\end{equation}
where $\chi_K$ is the convex characteristic function with value $0$ on $K$ and $+\infty$ elsewhere, and $\overline K$ is the strong $L^d$ closure of $K$. This closure in turn satisfies (as stated in \cite[Def.~2.2]{ChaCasCreNovPoc10} and proved in \cite[Prop.~7]{BreHol16}) the identity
\[\partial \TV(0)=\overline K = \left\{ \div z \,\middle\vert\, z \in L^\infty(\Omega; \R^d),\, \|z\|_{L^\infty(\Omega)}\ls 1,\, \div z \in L^d(\Omega), \, z \cdot n_{\partial \Omega} = 0\text{ on }\partial \Omega \right\},\]
where the equality $z \cdot n_{\partial \Omega} = 0$ is understood in the sense of the normal trace in $W^{1,d}(\div)$, that is,
\[\int_\Omega \div z \,u = - \int_\Omega z \cdot \nabla u \text{ for all } u \in W^{1, d/(d-1)}(\Omega).\]
Now, since $\TV$ is positively one-homogeneous, we have (see \cite[Lem.~A.1]{IglMerSch18}, for example) that
\begin{equation}\label{eq:TVsubgrad1}v \in \partial \TV (u)\text{ whenever }v \in \partial \TV(0) \text{ and }\int_\Omega vu = \TV(u).\end{equation}
It is natural to ask whether one can use that $v = \div z$ for some $z$ to integrate by parts in the last equality, which would formally lead to $z \cdot Du=|Du|$, or $z = Du/|Du|$. However, since $z$ is only in $L^\infty(\Omega;\R^d)$ and not guaranteed to be continuous, there is no immediate meaning to the action of the measure $Du$ on $z$, and $Du/|Du|$ can only be defined $|Du|$-a.e.~as the polar decomposition of $Du$. This difficulty is mitigated by defining (as first done in \cite{Anz83}) the product between $Du$ and $z$ as a distribution $(z, Du)$ given by
\begin{equation}\label{eq:anzellottidef}\big\langle(z,Du),\varphi\big\rangle = -\int_\Omega (z \cdot \nabla \varphi) u - \int_\Omega u \varphi \div z  \text{ for }\varphi \in C^\infty_c(\Omega).\end{equation}
This definition makes sense whenever $u \in \BV(\Omega) \cap L^{d/(d-1)}(\Omega)$ and 
\[z \in X_d(\Omega)= \left\{ z \in L^\infty(\Omega; \R^d)\,\middle\vert\, \div z \in L^d(\Omega)\right\} \supset \partial \TV(0).\]
Moreover if $\supp \varphi \subset A$ for some open set $A$ also \cite[Thm.~1.5]{Anz83}
\[\big|\big\langle(z,Du),\varphi\big\rangle\big| \ls \|z\|_{L^\infty(\Omega)} \|\varphi\|_{L^\infty(\Omega)}|Du|(A),\]
so it can be extended to a measure, which in fact is absolutely continuous with respect to $|Du|$, and whose corresponding Radon-Nikod\'ym derivative we denote by $\theta(z, Du, \cdot) \in L^1(\Omega, |Du|)$. Moreover, there exists a generalized normal trace of $z$ on $\partial\Omega$ denoted by $[z, n_{\partial \Omega}] \in L^\infty(\partial \Omega)$ for which the following Green's formula \cite[Thm.~1.9]{Anz83} holds:
\begin{equation}\label{eq:greenthm}\int_{\partial \Omega} [z, n_{\partial \Omega}] \tilde u \dd\mathcal{H}^{d-1} = (z, D\tilde u)(\Omega) + \int_\Omega \div z\, \tilde u\text{ for all }\tilde u \in \BV(\Omega) \cap L^{d/(d-1)}(\Omega).\end{equation}
With these definitions in mind \eqref{eq:TVsubgrad1} may be written \cite[Prop.~1.10]{AndCasMaz04} as $v \in \partial \TV (u)$ if and only if
\begin{equation}\label{eq:TVsubgrad2}\begin{gathered}v = -\div z \text{ for some } z \in X_d(\Omega) \text{ satisfying }\\ \|z\|_{L^\infty(\Omega)}\ls 1,\, [z, n_{\partial \Omega}] = 0,\text{ and }(z, Du)(\Omega) = \TV(u).\end{gathered}\end{equation}

Since we will make extensive use of level sets, we now fix the notation we use for them:
\begin{definition}\label{def:levelsets}
Given a function $u \in L^{d/(d-1)}(\Omega)$ we denote by $E^s$ the upper level set $\{u >s\}$ if $s>0$, and the lower level set $\{u<s\}$ for $s<0$, so that $|E^s|<+\infty$ for all $s \in \R \setminus \{0\}$. These two cases can be summarized as
\[E^s := \big\{ x \,|\, \sign(s)\,u(x) > |s|\big\},\]
\end{definition}
Finally, the following characterization of the subgradient of $\TV$ in terms of perimeter of level sets is crucial for our results below:
\begin{lemma}\label{lem:subgslicing}
Let $u$ and $E^s$ be as in Definition \ref{def:levelsets}. The following assertions are equivalent:
 \begin{itemize}
  \item\label{i:1} $v \in \partial \TV(u) \subset L^d(\Omega)$.
  \item\label{i:2} $v \in \partial \TV(0)$ and $\int_{\Omega} vu = \TV(u)$.
  \item\label{i:3} $v \in \partial \TV(0)$ and for a.e.~$s$,
  \begin{equation}\label{eq:perequality}\per(E^s) =\operatorname{sign}(s) \int_{E^s} v.\end{equation}
  \item\label{i:4} For a.e.~$s$, the level sets $E^s$ satisfy
  \[E^s \in \argmin_{E \subset \Omega} \per(E) - \operatorname{sign}(s) \int_E v.\]
 \end{itemize} 
\end{lemma} 
\begin{proof}The equivalence between the first two items follows from $\TV$ being positively one-homogeneous, and a proof can be found in \cite[Lem.~A.1]{IglMerSch18}, for example. A proof for the other statements in the $L^2(\R^2)$ setting, but which generalizes without modification to the present case, can be found in \cite[Prop.\ 3]{ChaDuvPeyPoo17}.
\end{proof}

\subsection{Dual solutions and their stability in one-homogeneous regularization}\label{sec:dualstuff}
In \cite[Prop.~3.3]{IglMer20}, the dual problem to \eqref{eq:primalproblem} is computed to be
\begin{equation}\label{eq:dualalphaw}\sup_{\substack{p \in L^{q\prime{}}(\Sigma), \\ A^\ast p \in \partial \TV(0) \subset L^d(\Omega)}} \scal{p}{f+w}_{(L^{q\prime}, L^q)}-\frac{\alpha^{\frac{1}{\sigma-1}}}{\sigma'}\|p\|_{L^{q\prime}(\Sigma)}^{\sigma'},\tag{$D_{\alpha,w}$}\end{equation}
where $q'=q/(q-1)$ is the conjugate exponent of $q$, and analogously for $\sigma'$, and we notice that owing to the strict concavity of the objective, this problem has a unique maximizer. Also, in \cite[Prop.~3.6]{IglMer20} it is proved that the solutions of \eqref{eq:dualalphaw} satisfy
 \begin{equation}\label{eq:parameterest}
\|p_{\alpha, w}-p_{\alpha, 0}\|_{L^{q\prime }(\Sigma)} \ls \rho_{L^q,\sigma}\left(\frac{\|w\|_{L^q(\Sigma)}}{2\alpha^{\frac{1}{\sigma-1}}}\right),
\end{equation}
where $\rho_{L^q,\sigma}$ is defined as the inverse of the function $\R^+ \to \R^+$ defined by
\begin{equation}t \mapsto \frac{\delta_{\|\cdot\|^{\sigma'}_{L^{q\prime}}\big/\sigma'}(t)}{t}\end{equation}
where $\delta_{\|\cdot\|^{\sigma'}_{L^{q\prime}}\big/\sigma'}$ is in turn the largest modulus of uniform convexity of the functional $\|\cdot\|^{\sigma'}_{L^{q\prime}(\Sigma)}\big/\sigma'$, that is, the largest function $\psi$ satisfying for all $u,v \in L^{q\prime}(\Sigma)$ and $\lambda \in (0,1)$:
\[\frac{1}{\sigma'}\|\lambda u+ (1-\lambda)v\|^{\sigma'}_{L^{q\prime}(\Sigma)} \ls  \frac{\lambda}{\sigma'}\|u\|^{\sigma'}_{L^{q\prime}(\Sigma)} + \frac{1-\lambda}{\sigma'}\|v\|^{\sigma'}_{L^{q\prime}(\Sigma)} - \lambda(1-\lambda)\psi(\|u-v\|_{L^{q\prime}(\Sigma)}).\]
Now, since we have chosen $\sigma = \min(q,2)$, we have \cite[Thm.~5.4.6, Ex.~5.4.7]{BorVan10} and \cite[Example 2.47]{SchKalHofKaz12}
$$\delta_{L^{q\prime}}(\epsilon) \gs C \epsilon^{\max(2, q/(q-1))} \implies \delta_{\|\cdot\|^{\sigma'}_{L^{q\prime}}\big/\sigma'}(t) \gs C t^{\sigma'}$$
where $\delta_{L^{q\prime}}$ denotes the modulus of uniform convexity of $L^{q\prime}(\Sigma)$ (which is independent of $\Sigma$) defined using points in the unit ball, that is
\[\delta_{L^{q\prime}}(\epsilon) = \inf \left\{ 1 - \left\|\frac{p_1+p_2}{2}\right\|_{L^{q\prime}(\Sigma)} \,\middle\vert\, \|p_1\|_{L^{q\prime}(\Sigma)}=\|p_2\|_{L^{q\prime}(\Sigma)}=1 \text{ and }\|p_1 - p_2\|_{L^{q\prime}(\Sigma)} \gs \epsilon \right\}.\]
In this setting and defining $v_{\alpha, w}:= A^\ast p_{\alpha, w} \in \partial \TV(u_{\alpha, w})$ as in \eqref{eq:optimality}, we can deduce from \eqref{eq:parameterest} that
\begin{equation}\label{eq:stability}
\|v_{\alpha, w}-v_{\alpha, 0}\|_{L^d(\Omega)} \ls C \|A^\ast\| \left(\frac{\|w\|_{L^q(\Sigma)}}{2\alpha^{\frac{1}{\sigma-1}}}\right)^{\frac{1}{\sigma'-1}} = C_{q,\sigma} \|A^\ast\| \left(\frac{\|w\|_{L^q(\Sigma)}^{\frac{1}{\sigma'-1}}}{\alpha^{\frac{1}{\sigma-1}\frac{1}{\sigma'-1}}}\right)= C_{q,\sigma} \|A^\ast\| \frac{\|w\|_{L^q(\Sigma)}^{\sigma-1}}{\alpha},
\end{equation}
which tells us that the threshold to control the effect of the noise on the dual variables is a direct generalization of the linear parameter choice that would arise in the case $\sigma=q=2$.

To know more about the asymptotic behavior of $v_{\alpha,w}$ as both $\alpha$ and $w$ vanish one needs an additional assumption. Whenever $u^\dag$ denotes a solution of $Au = f$ of minimal $\TV$ among such solutions, we say that the source condition is satisfied if there exists $p \in L^{q\prime}(\Omega)$ such that
\begin{equation}\label{eq:sourcecond}A^\ast p \in \partial \TV(u^\dag).\end{equation}
Now, since $\TV$ is not strictly convex, $u^\dag$ may not be unique, but arguing as in \cite[Rem.~3]{IglMerSch18} we have that if $\hat{u}^\dag$ is another such solution and we have \eqref{eq:sourcecond}, then $A^\ast p \in \partial \TV(\hat{u}^\dag)$ as well, for the same source element $p$. For our purposes, the significance of this source condition comes from the fact that it guarantees that the formal dual problem $(D_{0,0})$ obtained by setting $\alpha=0$ and $w=0$ in \eqref{eq:dualalphaw} has at least one maximizer, namely the source element $p$.

Under this source condition assumption and taking a sequence $\alpha_n \to 0$ we have in the case where noise is not present ($w=0$) the convergence 
\begin{equation}\label{eq:noiselessconv}v_{\alpha_n, 0} = A^\ast p_{\alpha, 0} \to A^\ast p_{0,0} =: v_{0,0} \text{ strongly in }L^d(\Omega),\end{equation}
where $p_{0,0}$ is the element of $L^{q\prime}(\Omega)$ of minimal norm among those satisfying \eqref{eq:sourcecond}, which is unique. This convergence is proved (see \cite[Prop.~3.4]{IglMer20} for a detailed argument) by testing the optimality of $p_{\alpha,0}$ and $p_{0,0}$ in $(D_{\alpha, 0})$ and $(D_{0, 0})$ with respect to each other to obtain both $p_{\alpha, 0} \rightharpoonup p_{0,0}$ weakly in $L^{q'}(\Sigma)$ and $\|p_{\alpha, 0}\|_{L^{q'}(\Sigma)} \ls \|p_{0,0}\|_{L^{q'}(\Sigma)}$, which together imply strong convergence by using the Radon-Riesz property of $L^{q'}(\Sigma)$. Finally, combining this convergence with \eqref{eq:stability}, we also have that if $w_n \in L^q(\Sigma)$ are such that
\begin{equation}\label{eq:noisyconv}\frac{\|w_n\|_{L^q(\Sigma)}^{\sigma-1}}{\alpha_n} \to 0,\text{ then also } v_{\alpha_n, w_n} \to A^\ast p_{0,0}\text{ strongly in }L^d(\Omega).\end{equation}

\section{Boundedness, uniform boundedness, and strong convergence}\label{sec:boundedness}
We start with a direct proof of boundedness of minimizers $u_{\alpha,w}$ of \eqref{eq:primalproblem} which, using the results cited in the previous section, can be made uniform for small noise and strong enough regularization. It relies on studying the level sets 
\[E^s_{\alpha, w} := \big\{ x \,|\, \sign(s)\,u_{\alpha,w}(x) > |s|\big\}.\]

\begin{prop}\label{prop:single_u_bounded}
Let $\Omega=\R^d$. Then minimizers $u_{\alpha,w}$ of \eqref{eq:primalproblem} belong to $L^\infty(\R^d)$. Moreover, under the source condition $A^\ast p_0 \in \partial \TV(u^\dag)$ and a parameter choice satisfying the condition 
\begin{equation} \label{eq:parameterchoice}\Vert w \Vert_{L^q(\Sigma)}^{\sigma-1} \ls \frac{\eta \alpha}{C_{q,\sigma} \|A^\ast\|}\end{equation}
where $C_{q,\sigma}$ is the constant in \eqref{eq:stability} and $\eta < \Theta_d$ is the constant of the isoperimetric inequality $\per(E) \gs \Theta_d |E|^{(d-1)/d}$ for $E$ of finite perimeter, we have a bound for $\|u_{\alpha,w}\|_{L^\infty(\R^d)}$ which is uniform in $\alpha$ and $w$. 
\end{prop}
\begin{proof}
We first show the claim for fixed $\alpha,w$. To do this, we make use of \eqref{eq:dualalphaw}: the problem has a solution $p_{\alpha,w}$, we have strong duality and the optimality condition
\begin{equation} 
v_{\alpha,w} := A^\ast p_{\alpha,w}  \in \partial \TV(u_{\alpha,w}). \label{eq:optcond}
\end{equation}
In what follows, we assume  without loss of generality that $s>0$ so that $E_{\alpha, w}^s := \{u_{\alpha,w} >s\}$. From \eqref{eq:perequality} and the H\"older inequality we can derive for a.e.~$s>0$ the estimate 
\[\per(E_{\alpha, w}^s) = \int_{E_{\alpha, w}^s} v_{\alpha,w} \ls |E_{\alpha, w}^s|^{(d-1)/d} \left( \int_{E_{\alpha, w}^s} |v_{\alpha,w}|^d \right)^{1/d}.\]
Using the isoperimetric inequality $\per(E_{\alpha, w}^s) \gs \Theta_d |E_{\alpha, w}^s|^{(d-1)/d}$, we obtain
\begin{equation}\label{eq:useisoper} \per(E_{\alpha, w}^s) \ls \frac{1}{\Theta_d}\per(E_{\alpha, w}^s) \left( \int_{E_{\alpha, w}^s} |v_{\alpha,w}|^d \right)^{1/d}.\end{equation}
Now, $v_{\alpha,w} \in L^d(\R^d)$ so for any $\eps>0$, there exists a $\delta>0$ such that for sets $E$ with $|E| \ls \delta$, $\int_E |v_{\alpha,w}|^d \ls \eps^d.$ In particular, if $|E_{\alpha, w}^s| \ls \delta$, it implies
\[ \per(E_{\alpha, w}^s) \ls \frac{\eps}{\Theta_d} \per(E_{\alpha, w}^s),\]
which is not possible for $\eps$ too small if $|E_{\alpha, w}^s| >0$.
Note that the bound on $\delta$ does not depend on $s$, only on $v_{\alpha,w}$, which means that we have a uniform positive lower bound on the mass of \emph{every} level set $E_{\alpha, w}^s$.
From the layer-cake formula for $L^{d/(d-1)}(\R^d)$ functions \cite[Thm.~1.13]{LieLos01}
\[\int_{\R^d} (u_{\alpha,w}^+)^{d/(d-1)} = \int_0^\infty \frac{d}{d-1} s^{1/(d-1)} |E_{\alpha, w}^s| \dd s, \text{ where } u_{\alpha,w}^+ = \max(u_{\alpha,w}, 0) \] we conclude that there must exist $s_0>0$ such that for $s \gs s_0$, $|E_{\alpha, w}^s| = 0$, which means that $u_{\alpha,w} \ls s_0$ a.e. in $\R^d$. We prove similarly that $u_{\alpha,w}^- = \max(0, - u_{\alpha,w})$ is bounded.

Finally, let us assume that $(\alpha_n,w_n)$ is a sequence of regularization parameters and noises for which \eqref{eq:parameterchoice} holds. From $v_{\alpha_n,0} \to v_{0,0}$ for any sequence $\alpha_n \to 0$ (see \eqref{eq:noiselessconv}), one infers that the family $\{v_{\alpha,0}\}_{\alpha > 0}$ is equi-integrable in $L^d(\R^d)$. We want, as before, to estimate $\int_{E_{\alpha, w}^s} |v_{\alpha,w}|^d.$ This can be done by writing
\begin{align*}\left(\int_{E_{\alpha, w}^s} |v_{\alpha,w}|^d \right)^{1/d}&\ls \left(\int_{E_{\alpha, w}^s} |v_{\alpha,0}|^d \right)^{1/d} + \left(\int_{E_{\alpha, w}^s} |v_{\alpha,w}-v_{\alpha,0}|^d \right)^{1/d} \\
&\ls \left(\int_{E_{\alpha, w}^s} |v_{\alpha,0}|^d \right)^{1/d} + \Vert v_{\alpha,w} - v_{\alpha,0} \Vert_{L^d(\R^d)}. \end{align*}
Now, using the equi-integrability of ${v_{\alpha,0}}$, one can, for any $\eps$, find $\delta$ so that as soon as $|E_{\alpha, w}^s|\ls \delta$, the first term of the right hand side is bounded by $\eps.$ On the other hand, the second term is, independently from $\delta$, bounded by $\eta < \Theta_d$, as a consequence of \eqref{eq:stability} and \eqref{eq:parameterchoice}. We conclude that as soon as $|E_{\alpha, w}^s| \ls \delta$, we have 
\[ \per(E_{\alpha, w}^s) \ls \frac{\eps + \eta}{\Theta_d} \per(E_{\alpha, w}^s)\]
which is still not possible for $\eps$ too small (independent of $s$ and $\alpha,w$ satisfying \eqref{eq:parameterchoice}) if $|E^s_{\alpha, w}|>0$.
\end{proof}
\begin{remark}By a similar argument on the $E_{\alpha, w}^s$ (see \cite[Lem.~5]{IglMerSch18}), we can actually show that $u_{\alpha,w}$ has compact support, so that $u_{\alpha,w}^+ \in L^1(\R^d)$ and we could use the (simpler) layer cake formula in $L^1$:
\[ \int_{\R^d} u_{\alpha,w}^+ = \int_0^\infty |E_{\alpha, w}^s| \dd s\]
which would provide the same contradiction.
\end{remark}

\begin{remark}\label{rem:intermediatecase}
The parameter choice condition \eqref{eq:parameterchoice} does not necessarily imply convergence of the dual variables. This is the case when
\[\frac{\Vert w_n \Vert_{L^q(\Sigma)}^{\sigma-1}}{\alpha_n} \to 0 \text{ as }n\to \infty\]
so that, as remarked in Section \ref{sec:dualstuff}, in fact \eqref{eq:noiselessconv} and \eqref{eq:stability} imply that 
\[ v_{\alpha_n,w_n} \to A^\ast p_0 \text{ strongly in }L^d(\Omega).\]
In particular, this implies that the family $(v_{\alpha_n,w_n})_n$ is equi-integrable in $L^d(\R^d)$, which means that in the reasoning coming after \eqref{eq:useisoper}, the $\delta$ can be chosen independent of $n$, which would simplify the proof in this more restrictive case.
\end{remark}

\begin{remark}\label{rem:denoisingcase}
In the case of denoising in the plane with the Rudin-Osher-Fatemi model \cite{RudOshFat92}, for which $d=m=2$, $q=2$ and $A=\mathrm{Id}:L^2(\R^2)\to L^2(\R^2)$, we have that $v_{\alpha, w} = \alpha^{-1}(f+w-u_{\alpha, w})$ and Proposition \ref{prop:single_u_bounded} applies as soon as $f+w \in L^2(\R^2)$.
\end{remark}

\begin{remark}\label{rem:nonemptysubg}
The same proof shows that if the source condition \eqref{eq:sourcecond} is satisfied, then $u^\dag \in L^\infty(\R^d)$.
\end{remark}

We consider now the case when $\Omega$ is a bounded Lipschitz domain, which leads to two cases for the functional in \eqref{eq:primalproblem}. First, we can consider the total variation $\TV(u;\R^d)$ of the extension by zero of $u$ from $\Omega$ to the whole $\R^d$, which we refer to as homogeneous Dirichlet boundary conditions. Second, we can consider the variation $\TV(u;\Omega)$ which we refer to as homogeneous Neumann boundary conditions, since the test functions in \eqref{eq:TVdef} are compactly supported in $\Omega$. 

\begin{prop}\label{prop:bounded_domains}
Proposition \ref{prop:single_u_bounded} also holds for bounded Lipschitz domains $\Omega$ and either homogeneous Dirichlet or Neumann boundary conditions on $u_{\alpha, w}$. 
\end{prop}
\begin{proof}
In the Dirichlet case we work with the functional in \eqref{eq:primalproblem} but consider $\TV(\cdot;\R^d)$ among functions in $L^{d/(d-1)}(\R^d)$ constrained to vanish on $\R^d \setminus \Omega$. This can be translated to a formulation on $L^{d/(d-1)}(\Omega)$ by considering $\mathcal{E}:L^{d/(d-1)}(\Omega) \to L^{d/(d-1)}(\R^d)$ the extension by zero on $\R^d \setminus \Omega$ and $\TV(\cdot\, ; \R^d) \circ \mathcal E$. By \cite[Cor.~3.89]{AmbFusPal00} and using that $\Omega$ is a Lipschitz domain we have that 
\[\TV( \mathcal{E}u; \R^d) = \TV(u;\Omega) + \int_{\partial \Omega} |u| \dd\mathcal{H}^{d-1},\]
where we emphasize that the first total variation is computed in $\R^d$ whereas the second one in $\Omega$ as in \eqref{eq:TVdef}. The values of $u$ at the boundary are understood in the sense of traces.

In this situation, since $\mathcal{E}$ is linear we can consider the composition $\TV \circ \,\mathcal{E}$ as a convex positively homogeneous functional in its own right, which gives us that indeed,
\begin{equation}\label{eq:optimalitydirichlet}
v_{\alpha,w} := A^\ast p_{\alpha, w} \in \partial_{L^{d/(d-1)}(\Omega)} (\TV \circ \,\mathcal{E})(u_{\alpha,w}) \subset L^d(\Omega),
\end{equation}
where $p_{\alpha, w}$ is defined in \eqref{eq:optimality}. By general properties of one-homogeneous functionals we then have (as in the first two items of Lemma \ref{lem:subgslicing}) that
\[\TV( \mathcal{E}u_{\alpha,w}; \R^d) = \int_\Omega v_{\alpha,w} u_{\alpha,w}.\]
Moreover, we can argue exactly as in \cite[Prop.~3]{ChaDuvPeyPoo17} (which uses more properties of the subgradient than the equality above) to obtain that also in this case
\begin{equation}\label{eq:perimeterequality}\per(E_{\alpha, w}^s; \R^d) = \sign(s) \int_{E_{\alpha, w}^s} v_{\alpha, w},\end{equation}
where the level sets $E_{\alpha, w}^s$ are defined as in Proposition \ref{prop:single_u_bounded}, and from which we can follow the rest of the proof exactly, since $E_{\alpha, w}^s \subset \Omega$.

In the Neumann case we consider $\TV$ as $\TV(u; \Omega)$. This parallels what is done in \cite[Sec.~6]{IglMer20} and with more details in the 2D case in \cite[Sec.~4.3]{IglMerSch18}. In this case, one uses the estimate (see \cite[Sec.~4.3]{IglMerSch18} for a proof)
\begin{equation}\label{eq:relativeisop} C_\Omega \per(F ; \Omega) \gs \frac{|F|^{\frac{d-1}{d}} |\Omega \setminus F|^{\frac{d-1}{d}}}{|\Omega|^{\frac{d-1}{d}}},\end{equation}
where $C_\Omega$ is the constant of the Poincar\'e-Sobolev inequality
\begin{equation}\label{eq:sobolevineq}\left\|u - \frac{1}{|\Omega|}\int_\Omega u\right\|_{L^{d/(d-1)}(\Omega)} \ls C_\Omega \,\TV(u\, ;\, \Omega).\end{equation}
To see that \eqref{eq:relativeisop} can play the role of the isoperimetric inequality, notice that we only need to use this inequality for large values of $|s|$ and sets of small measure. Therefore we may assume
\[|E^s_{\alpha, w}| \ls \frac{|\Omega|}{2}, \text{ so that } C_\Omega \per(E^s_{\alpha, w} ; \Omega) \gs \frac{1}{2^{\frac{d-1}{d}}}|E^s_{\alpha, w}|^{\frac{d-1}{d}}.\qedhere\]
\end{proof}

\begin{remark}
Let us assume that $A$ and $A^\ast$ preserve boundedness, that is, $Au \in L^\infty(\Sigma)$ for all $u \in L^\infty(\Omega)$ and $A^\ast p \in L^\infty(\Omega)$ whenever $p \in L^\infty(\Sigma)$. Then, if $f,w \in L^\infty(\Sigma)$ and in the situation of Proposition \ref{prop:single_u_bounded}, the optimality condition \eqref{eq:optimality} implies that $v_{\alpha, w}\in L^\infty(\Omega)$ as well. One can then use strong regularity results (see \cite[Thm.~3.1]{Mas75} or \cite[Thm.~3.6(b)]{GonMas94}, for example) to obtain that $\partial\{u_{\alpha,w} >s\} \in C^{1,\gamma}$ for all $\gamma < 1/4$, provided that $d\ls 7$. Moreover, since our estimates are uniform under the assumptions of Theorem \ref{thm:joint_thm}, this regularity can also be made uniform along a sequence \cite[Secs.~1.9, 1.10]{Tam84} as well.

The assumption of both $A$ and $A^\ast$ preserving boundedness is easily seen to be satisfied for convolution operators with regular enough kernels, since in this case $A^\ast$ is of the same type. However, it also holds for other commonly used operators. As an example, let us consider the case of the Radon transform $\mathcal{R}$ of functions on a bounded domain $\Omega \subset \R^d$. In this case one can consider $\Sigma = \mathbb{S}^{d-1} \times (-R, R)$ with $R>0$ large enough so that $\Omega \subset B(0,R)$ and
\[A: L^{d/(d-1)}(\Omega) \subset L^{2d/(2d-1)}(\R^d) \xrightarrow{\,\mathcal{R}\,} L^2(\mathbb{S}^{d-1} \times \R) \xrightarrow{ r_{\mathbb{S}^{d-1}\times (-R,R)}} L^2\big(\mathbb{S}^{d-1}\times (-R,R)\big),\]
the last map being restriction, since $R$ is continuous between the middle two spaces \cite[Thm.~1]{ObeSte82}. Now, if $u \in L^\infty(\Omega)$ then clearly $Au \in L^\infty\big(\mathbb{S}^{d-1} \times (-R,R)\big)$ as well, since the integrals on planes in the definition of $\mathcal{R}$ are then on  domains of uniformly bounded Hausdorff measure $\mathcal{H}^{d-1}$, that is 
\[\left|Au(\theta, t)\right|=\left|\int_{\theta^\perp} u(t\theta + y) \dd\mathcal{H}^{d-1}(y) \right| \ls \|u\|_{L^\infty(\Omega)}\,\mathcal{H}^{d-1}\big( B(0,R)\cap\{x_1=0\}\big)\, \text{ for a.e. }(\theta, t).\]
Similarly, $A^\ast$ is then extension by zero composed with the backprojection integral operator, so we have
\[|A^\ast p(x)| = \left|\int_{\mathbb{S}^{d-1}} p(\theta, x \cdot \theta) \dd \mathcal{H}^{d-1}(\theta)\right| \ls \|p\|_{L^\infty(\mathbb{S}^{d-1}\times(-R, R))}\mathcal{H}^{d-1}(\mathbb{S}^{d-1})\, \text{ for a.e. }x.\]

Observe that for the case of $\Omega\subset\R^2$ and convolutions with $L^2$ kernels boundedness is immediate, since Young's inequality for convolutions used in \eqref{eq:optimality} directly implies $v_{\alpha, w}\in L^\infty(\Omega)$ as soon as $u_{\alpha, w}, f, w \in L^2(\Omega)$.
\end{remark}

\begin{cor}\label{cor:strong_conv}
Under the assumptions of Proposition \ref{prop:single_u_bounded}, for a sequence of minimizers $\{u_{\alpha_n, w_n}\}_n$ with $\alpha_n, w_n$ satisfying \eqref{eq:parameterchoice} and $\alpha_n \to 0$ and $w_n \to 0$ as $n \to \infty$, we have that, up to a subsequence,
\begin{equation}\label{eq:strongconv}u_{\alpha_n, w_n} \xrightarrow[n \to \infty]{} u^\dag\text{ strongly in }L^{\overline p}(\Omega)\text{ for all }\overline p \in (1, \infty),\end{equation}
where $u^\dag$ is an exact solution of $Au=f$ with minimal $\TV$ among such solutions. If there is only one exact solution, then the whole sequence converges to it in the same fashion.
\end{cor}
\begin{proof}
We first notice that the parameter choice \eqref{eq:parameterchoice} is less restrictive than the one needed in \cite[Prop.~3.1]{IglMer20} which provides strong convergence to some $u^\dag$ in $L^{\hat p}_\loc(\Omega)$ for $\hat{p} \in (1, d/(d-1))$ and up to a subsequence by a basic compactness argument. Moreover, if $\Omega=\R^d$, we can apply \cite[Lem.~5.1]{IglMer20} to obtain that all of the $u_{\alpha_n, w_n}$ and $u^\dag$ are supported inside a common ball $B(0,R)$ for some $R>0$. If, in contrast, $\Omega$ is bounded, then $\Omega \subset B(0,R)$ and we may extend to the latter by zero.

For this subsequence (which we do not relabel) and $\hat{p} \ls \overline p$ we have
\begin{equation}\label{eq:hoelderbound}\int_\Omega |u_{\alpha_n, w_n} - u^\dag|^{\overline p} \ls \|u_{\alpha_n, w_n} - u^\dag\|_{L^{\hat p}(B(0,R))}^{\hat{p}} \|u_{\alpha_n, w_n} - u^\dag\|_{L^\infty(\Omega)}^{{\overline p}-\hat{p}},\end{equation}
which using Proposition \ref{prop:single_u_bounded} immediately implies \eqref{eq:strongconv}.

Finally, if the minimal $\TV$ solution of $Au=f$ is unique, any subsequence has a further subsequence converging to this unique solution, so the whole sequence must in turn converge to it.
\end{proof}

\begin{remark}
In the plane, under the same source condition \eqref{eq:sourcecond}, a parameter choice equivalent to the one we use for $d=2$, and additional assumptions for both $A$ and $u^\dag$, linear convergence rates are proved in the $L^2$ setting in \cite[Thm.~4.12]{Val21}, that is $\|u_{\alpha_n, w_n} - u^\dag\|_{L^2(\Omega)} = O(\alpha_n)$ (observe that the parameter choice forces $\|w_n\|=O(\alpha_n)$). Therefore, this result can be combined with ours to obtain also a convergence rate of order $2/\overline p$ in $L^{\overline p}(\Omega)$. To see this, we just notice that by Remark \ref{rem:nonemptysubg} we have $u^\dag \in L^\infty(\Omega)$ in addition to the uniform bound on $\|u_{\alpha_n, w_n}\|_{L^\infty(\Omega)}$, so that as in \eqref{eq:hoelderbound} we have
\[\|u_{\alpha_n, w_n} - u^\dag\|_{L^{\overline p}(\Omega)} \ls \|u_{\alpha_n, w_n} - u^\dag\|_{L^2(\Omega)}^{2/\overline{p}} \|u_{\alpha_n, w_n} - u^\dag\|_{L^\infty(\Omega)}^{(\overline {p}-2)/\overline p} = O(\alpha_n^{2/\overline{p}}).\]
\end{remark}

One might wonder if strong convergence in $L^\infty(\Omega)$ is possible. In fact, it is not:

\begin{example}
Consider (as in Remark \ref{rem:denoisingcase}) denoising in the plane $\R^2$ with parameter $\alpha_n = 1/n$ of $f+w_n$, where for arbitrarily small $c >0$ we define
\begin{gather*}u_{0,0}=f=1_{B(0,1)}\text{ and } w_n = 1_{B\left(0,\sqrt{1 + \frac{c}{n}}\right)} - 1_{B(0,1)},\\ \text{ so that }\|w_n\|_{L^2(\R^2)}=\frac{\sqrt{\pi}c}{n} \text{ and } \frac{\|w_n\|_{L^2(\R^2)}}{\alpha_n} = \sqrt{\pi} c.\end{gather*}
Then since $f+w_n = 1_{B\left(0,\sqrt{1+\frac{c}{n}}\right)}$ we know \cite[Sec.~2.2.3]{ChaCasCreNovPoc10} that $u_{\alpha_n,w_n}$ is also proportional to $1_{B\left(0,\sqrt{1+\frac{c}{n}}\right)}$, making $L^\infty(\R^2)$ convergence impossible. Notice that this situation does not change for a more aggressive parameter choice.
\end{example}

\subsection{Results with density estimates}
The combination of the source condition $\operatorname{Ran}(A^\ast) \cap \partial \TV(u^\dag) \neq \emptyset$ and the parameter choice \eqref{eq:parameterchoice} leads to uniform weak regularity estimates for the level sets of $u_{\alpha, w}$, and boundedness may also be deduced from those. More precisely, recalling that
\[E^s_{\alpha, w} = \big\{ x \,|\, \sign(s)\,u_{\alpha,w}(x) > |s|\big\},\]
it is proved in \cite[Thm.~4.5]{IglMer20} that
\begin{equation}\label{eq:densityest}|B(x,r) \cap E^s_{\alpha, w}| \gs C |B(x,r)|\text{ and } |B(x,r) \setminus E^s_{\alpha, w}| \gs C |B(x,r)|,\end{equation}
for $x \in \partial E^s_{\alpha_, w}$ and $r \ls r_0$, where $r_0$ and $C$ are independent of $x$, $n$ and $s$, and where we have taken a representative of $E^s_{\alpha, w}$ for which the topological boundary equals the support of the derivative of its indicatrix \cite[Prop.~12.19]{Mag12}, that is $\partial E^s_{\alpha, w} = \supp D1_{E^s_{\alpha, w}}$. In turn, this support can be characterized \cite[Prop.~12.19]{Mag12} by the property
\[\supp D1_{E^s_{\alpha, w}} =\left\{ x \in \R^d \,\middle\vert\, 0< \frac{|E^s_{\alpha, w} \cap B(x,r)|}{|B(x,r)|} < 1 \text{ for all } r>0\right\},\]
where we note that it could be that these quotients tend to $0$ or $1$ as $r\to0$.

We refer to the inequalities in \eqref{eq:densityest} as inner and outer density estimates respectively, and the combination of both as $E^s_{\alpha, w}$ satisfying uniform density estimates. Such estimates are the central tool for the results of convergence of level set boundaries in Hausdorff distance in \cite{ChaDuvPeyPoo17, IglMerSch18, IglMer20, IglMer21}.

The proof of \eqref{eq:densityest} is more involved than those in the previous section, but once it is obtained, boundedness of minimizers follows promptly:
\begin{prop}\label{prop:strong_conv_densityest}
Assume that the $E^s_{\alpha, w}$ satisfy uniform density estimates at scales $r \ls r_0$ with constant $C$, and that $\alpha$ is chosen in terms of $\|w\|_{L^q(\Sigma)}$ so that $u_{\alpha, w}$ is bounded in $L^{d/(d-1)}(\Omega)$ for some $p \gs 1$. Then, in fact $\{u_{\alpha,w}\}$ is uniformly bounded in $L^\infty(\Omega)$. If additionally we have a sequence $u_{\alpha_n, w_n} \to u^\dag$ strongly in $L^{\hat p}(\Omega)$ for some $\hat p$, then also $u_{\alpha_n, w_n} \to u^\dag$ strongly in $L^{\overline p}(\Omega)$ for all ${\overline p}$ such that $\hat p \ls {\overline p} < \infty$.
\end{prop}
\begin{proof}
If $u_{\alpha,w}$ is not bounded in $L^\infty(\Omega)$, then for every $M > 0$ there are $\alpha := \alpha(M)$ and $w:=w(M)$ for which $|E^M_{\alpha, w}|>0$. The inner density estimate at scale $r_0$ for some point $x \in \partial E^M_{\alpha, w}$ then reads
\[|B(x,r_0) \cap E^M_{\alpha, w}| \gs C |B(x,r_0)|,\]
which implies 
\[\int_\Omega|u_{\alpha, w}|^{d/(d-1)} \gs \int_{B(x,r_0) \cap E^M_{\alpha, w}}|u_{\alpha, w}|^{d/(d-1)} \gs C |B(x,r_0)| M^{d/(d-1)},\]
which, since $r_0$ is fixed, contradicts the fact that the family $\{u_{\alpha, w}\}$ is bounded in $L^{d/(d-1)}(\Omega)$. For the second statement, we can argue as in \eqref{eq:hoelderbound}.
\end{proof}

\begin{remark}
In the Dirichlet case of Proposition \ref{prop:bounded_domains} we have assumed only that $\Omega$ is a Lipschitz domain without further restrictions. In this context, unless $\Omega$ is convex, it is not true that $E^s$ is a minimizer of 
\[\per(E) - \int_{E} \mathcal{E} v_{\alpha, w}\text{ among }E \subset \R^d,\]
since these may extend beyond $\Omega$, while the functional $u \mapsto (\TV \circ \,\mathcal{E})(u) - \int_{\Omega} v_{\alpha, w} u$ is only sensitive to variations supported in $\Omega$. This kind of variational problem is used in \cite{IglMer20} and \cite{IglMerSch18} to obtain the density estimates \eqref{eq:densityest}, but for that one needs (see \cite[Lem.~9]{IglMerSch18}) to extend $v_{\alpha, w}$ by a variational curvature of $\Omega$ minorized by a function in $L^{d}(\R^d)$. The existence of such a curvature is an additional restriction on $\Omega$ and it is not satisfied for domains with inner corners in $\R^2$, for example.
\end{remark}

One can also make the bound slightly more explicit in terms of the constant and scale of the density estimates:

\begin{remark}
Let $u$ have level sets $E^s := \{\sign(s) u>|s|\}$ satisfying density estimates with constant $C$ at scale $r_0$. Then, we have 
\[\Vert u \Vert_{L^\infty} \ls\frac{1}{C^{1/p} \cdot |B(0,r_0)|^{1/p}} \Vert u \Vert_{L^p}.\]
This is a direct application of the Markov inequality. Indeed, since there is always a boundary point $x_0$, the level sets $E^s$ which are nonempty must satisfy $|E^s| \gs C |B(x_0,r_0)|=C|B(0,r_0)|,$ which implies
\[ \int_\Omega |u|^p \gs |s|^p |E^s| \gs  |s|^p C |B(0,r_0)|.\]
In particular, this implies
\[\frac{1}{C^{1/p}|B(0,r_0)|^{1/p}}\Vert u \Vert_{L^p} \gs |s|\]
meaning that as soon as $|s|$ exceeds $\frac{1}{C^{1/p} |B(0,r_0)|^{1/p}}\Vert u \Vert_{L^p}$, $E^s$ must be empty.
\end{remark}

\begin{cor}
Under the assumptions of Proposition \ref{prop:strong_conv_densityest} above, let $s$ with $|s| > \|u^\dag\|_{L^\infty}$. Then, we have
\[\limsup_{n \to \infty} \partial E^s_{\alpha_n,w_n} = \emptyset,\]
where $\limsup \partial E^s_{\alpha_n,w_n}$ consists \cite[Definition 4.1]{RocWet98} of all of limits of subsequences of points in $\partial E^s_{\alpha_n,w_n}$.
\end{cor}
\begin{proof}
By the convergence $u_{\alpha_n,w_n} \to u^\dag$ in $L^{\hat{q}}(\Omega)$, if $|s| > \|u^\dag\|_{L^\infty}$ we have $|E^s_{\alpha_n, w_n}| \to 0$. If we had $x \in \limsup_n \partial E^s_{\alpha_n,w_n}$, we could produce a subsequence $\{x_n\}_n$ in $\Omega$ with $x_n \in \partial E^s_{\alpha_n,w_n}$ and $x_n \to x$ for $x \in \overline\Omega$. Using the inner density estimate for some $r_0>0$, we would end up with
\[|E^s_{\alpha_n,w_n}| \gs |B(x_{\alpha_n}, r_0) \cap E^s_{\alpha_n,w_n}| \gs C |B(0,r_0)|,\]
which contradicts $|E^s_{\alpha_n, w_n}| \to 0$.
\end{proof}

\section{Taut string with weights and unbounded radial examples}\label{sec:string}
In this section we consider denoising of one-dimensional data with a modified Rudin-Osher-Fatemi (ROF) functional with weights in both terms, obtaining a taut string characterization of the solutions which reduces the problem to finding finitely many parameters. The main difficulty is that we allow weights which may degenerate at the boundary, which forces the use of weighted function spaces in all the arguments.

\subsection{One-dimensional weighted ROF problem and optimality condition}
We start with two weights $\phi, \rho$ on  the interval $(0,1)$ for which
\begin{equation}\label{eq:weightasusmptions}\phi, \rho >0 \text{ on }(0,1), \, \rho \in \mathrm{Lip(0,1)} , \, \phi \in C(0,1)\cap L^\infty(0,1),  \text{ and } \frac{1}{\phi} \in L^\infty_\loc(0,1),\end{equation}
where we notice that it could be that $\phi(x)\to 0$ or $\rho(x)\to 0$ as $x\to 0$ or $x \to 1$. Using these weights we consider the weighted denoising minimization problem 
\begin{equation}\label{eq:weightedROF}\inf_{ u \in L^2_\phi(0,1)} \int_0^1 \big(u(x)-f(x)\big)^2\phi(x) \dd x + \alpha \TV_\rho(u),\end{equation}
where we define the weighted total variation with weight $\rho$ in the usual way (see \cite{AmaBel94, Bal01}, for example) as
\begin{equation}\label{eq:weightedTV}\TV_\rho(u) := \sup\left\{ \int_0^1 u(x) z'(x) \dd x \, \middle\vert\, z \in C^1_c(0,1) \text{ with } |z(x)| \ls \rho(x) \text{ for all } x\right\}\end{equation}
and the problem is considered in the weighted Lebesgue space
\[L^2_\phi(0,1) := \left\{ u \, \middle\vert \, \|u\|^2_{L^2_\phi} := \int_0^1 u^2 \phi < \infty\right\},\]
which is a Hilbert space when considered with the inner product
\[\langle v , u\rangle_{L^2_{\phi}} = \int_0^1 \phi\, v u.\]
The predual variables and optimality conditions for \eqref{eq:weightedROF} will then naturally be formulated on a weighted Sobolev space, namely
\begin{equation}\label{eq:W11phi}W^{1,2}_{1,1/\phi}(0,1):=\left\{ U \in L^2(0,1) \, \middle \vert \, U' \in L^2_{1/\phi}(0,1) \right\}.\end{equation}

The first step is to see that under the assumptions \eqref{eq:weightasusmptions}, the denoising problem is still well posed:
\begin{prop}Assume that $f \in L^2_\phi(0,1)$. Then, there is a unique minimizer $u_\alpha$ of \eqref{eq:weightedROF}.
\end{prop}
\begin{proof}With the assumptions on $\phi, \rho$, these weights stay away from zero on any compact subset of $(0,1)$, which means that we have, for an open subset $A \subset (0,1)$, for which in particular $\overline A \subset \subset (0,1)$, a constant $C_A$ such that for every $u\in L^2_\phi(0,1)$
\[ \int_A \big(u(x)-f(x)\big)^2\phi(x) \dd x + \alpha \TV_\rho(u ; A) \gs  C_A\left( \int_A \big(u(x)-f(x)\big)^2 \dd x + \alpha \TV(u ; A)\right).\]
This means that along a minimizing sequence $\{u_n\}_n$ for \eqref{eq:weightedROF}, we have
\[ \sup_{n} \int_A \big(u_n(x)-f(x)\big)^2 \dd x + \alpha \TV(u_n ; A) < +\infty.\]
Applying Cauchy-Schwarz on the first term and since $\alpha >0$, we obtain
\[ \sup_{n} \int_A \big|u_n| + \alpha \TV(u_n ; A) < +\infty\]
which allows applying \cite[Thm.~3.23]{AmbFusPal00} to conclude that a not relabeled subsequence of $\{u_n\}_n$ converges in $L^1_{\text{loc}}(0,1)$ to some $u_\infty \in \BV_{\text{loc}}(0,1).$ Now, arguing as in \cite[Prop.~1.3.1]{Cam08} we have that $\TV_\rho$ is lower semicontinuous with respect to strong $L^1_{\text{loc}}$ convergence. To see this, let $z \in C^1_c(0,1)$ with $|z|\ls \rho$, so that we have $u_n \xrightarrow{L^1(\supp(z))} u_\infty$ and 
\[ \int_0^1 u_n z' \to \int_0^1  u_\infty z'.\]
Since the left hand side is bounded by $\TV_\rho(u_n)$, this implies in particular that for any such $z$,
\[ \liminf_n \TV_\rho(u_n) \gs \int_0^1 u_\infty z'.\]
Taking the supremum over $z$ as in \eqref{eq:weightedTV}, we obtain the semicontinuity
\begin{equation}\label{eq:TVrhoLSC} \liminf_n \TV_\rho(u_n) \gs \TV_\rho(u_\infty).\end{equation}
Moreover, we may take a further subsequence of $\{u_n\}_n$ converging weakly in $L^2_\phi(0,1)$ to a limit which clearly must be again $u_\infty$, and the first term of \eqref{eq:weightedROF} is lower semicontinuous with respect to this convergence since it involves only the squared norm of $L^2_\phi(0,1)$. This and \eqref{eq:TVrhoLSC} show that $u_\infty$ is actually a minimizer of \eqref{eq:weightedROF}.
\end{proof}

We have that $u$ realizes the infimum in \eqref{eq:weightedROF} if and only if
\begin{equation}\label{eq:weightedROFoptimality}\frac{f - u}{\alpha} \in \partial_{L^2_\phi}\TV_\rho(u) \end{equation}
where the subgradient means that $v \in \partial_{L^2_\phi}\TV_\rho (u)$ if and only if
\[\TV_\rho(u)  + \int_0^1 v (\tilde u - u) \phi \ls \TV_\rho(\tilde u) \text{ for all }\tilde u \in \BV_\rho(0,1) \cap L^2_\phi(0,1).\]

This set is characterized (with assumptions on the weights covering ours) in \cite[Lem.~2.4]{AthJerNovOrl17}, which in one dimension yields the following generalization of \eqref{eq:TVsubgrad2}:
\begin{equation}\label{eq:TVrhosubgrad}\begin{gathered}\text{There exists }\xi \in L^\infty(0,1)\text{ with } |\xi|\ls 1\text{ and }(\rho \xi)' \in L^2_{1/\phi}(0,1),\\\text{ for which }\int_0^1 v\tilde u \phi = (\rho\xi, D\tilde u)(0,1) \text{ for all }\tilde u \in \BV_\rho(0,1) \cap L^2_\phi(0,1) \\\text{ and also } \int_0^1 vu \phi = (\rho\xi, Du)(0,1) = \TV_\rho(u),\end{gathered}\end{equation}
where it is to be noted that in comparison to \eqref{eq:TVsubgrad2}, we always consider the product $\rho\xi$ together, and generally avoid differentiating $\rho$ alone. Like for \eqref{eq:anzellottidef}, the second line of \eqref{eq:TVrhosubgrad} generalizes integration by parts when $\tilde u \in C^\infty_c(0,1)$, so we have that $v\phi = -(\rho \xi)'$ in the sense of distributions, and $\rho \xi \in W^{1,2}_{1,1/\phi}(0,1)$. In analogy to \eqref{eq:greenthm}, it is tempting to interpret this equality in terms of boundary values. However, as remarked in \cite[Below Lemma 2.4]{AthJerNovOrl17}, to which extent this is possible depends on further properties of $\phi$. In our one-dimensional case and noting that because $\phi \in L^\infty(0,1) \cap C(0,1)$ the inverse $1/\phi$ remains bounded away from zero on the closed interval, and we have
\begin{equation}\label{eq:continuousembedding}W^{1,2}_{1,1/\phi}(0,1)\subset W^{1,2}(0,1) \subset C([0,1]),\end{equation}
and as we will prove in Theorem \ref{thm:weighteddual}, in fact we have that $\rho(0) \xi (0) = \rho(1) \xi(1) = 0$, in particular for $v = (f-u)/\alpha$ in \eqref{eq:weightedROFoptimality}.

Since $\TV_\rho$ is positively one-homogeneous, we recognize in the above that the first statement of \eqref{eq:TVrhosubgrad} characterizes $\partial_{L^2_\phi}\TV_\rho (0)$, and by Fenchel duality in $L^2_\phi(0,1)$ (as in \cite[Thm.~III.4.2]{EkeTem99} with $\Lambda=\mathrm{Id}$), we have 
\begin{multline}\label{eq:dualityweightedrof}\min_{ u \in L^2_\phi(0,1)} \int_0^1 \big(u(x)-f(x)\big)^2\phi(x) \dd x + \alpha \TV_\rho(u) \\ = \max_{v \in \partial_{L^2_\phi}\TV_\rho (0)} \frac{1}{\alpha}\langle v , f\rangle_{L^2_\phi} - \frac12 \|v\|^2_{L^2_\phi} = \max_{v \in \partial_{L^2_\phi}\TV_\rho (0)} \int_0^1 \left(\frac{vf}{\alpha} - \frac{v^2}{2}\right)\phi, \end{multline}
where the last maximization problem can be written as
\[\max_{v \in \partial_{L^2_\phi}\TV_\rho (0)} \int_0^1 \frac{1}{\alpha} vf\phi - \frac12 \|v\|^2_{L^2_\phi}= - \left[\min_{v \in \partial_{L^2_\phi}\TV_\rho (0)} \frac12 \Big\|\frac{f}{\alpha} - v\Big\|^2_{L^2_\phi} - \frac{1}{2\alpha^2}\|f\|^2_{L^2_\phi}\right],\]
for which we notice that since the last term does not involve $v$, is the familiar formulation as projection on a convex set.

For explicitly characterizing the minimizers, we will need to give a pointwise meaning to \eqref{eq:weightedROFoptimality}, which turns out to be
\begin{equation}\label{eq:stringoptimality}
\begin{gathered}\int_0^x \big(f(s)-u(s)\big)\, \phi(s)\dd s = -U(x) \text{ for all } x \in (0,1), \text{ where }\\
U \in W^{1,2}_{1,1/\phi, 0}(0,1) \text{ satisfies } |U| \ls \alpha \rho \text{ and } U(x) = \alpha \rho(x) \frac{Du}{|Du|}(x) \text{ for }|Du| \text{-a.e.~} x,
\end{gathered}\end{equation}
as we prove in Theorem \ref{thm:weighteddual} below. Here, avoiding density questions and owing to the embedding \eqref{eq:continuousembedding} we have directly defined 
\[W^{1,2}_{1,1/\phi, 0}(0,1) = \left\{ U \in W^{1,2}_{1,1/\phi}(0,1)\,\mid\, U(0)=U(1)=0\right\}.\]
A first remark about this characterization is that the last equality in the second line of \eqref{eq:stringoptimality} is made possible by the one-dimensional setting, since then functions in $W^{1,2}_{1,1/\phi, 0}(0,1)$ are continuous. Whether such a characterization is possible in higher dimensions, even without weights, is a much more delicate question, see \cite{BreHol16, ChaGolNov15, ComCraDecMal22}. A further remark is that \eqref{eq:stringoptimality} implies in particular
\begin{equation}\label{eq:notquiteoptimality}\int_0^x \big(f(s)-u(s)\big)\, \phi(s)\dd s = -\alpha \rho(x) \frac{Du}{|Du|}(x) \text{ for }|Du| \text{-a.e.~} x,\end{equation}
which is formulated only in terms of the minimizer $u$ and does not involve additional derivatives. However, equation \eqref{eq:notquiteoptimality} only provides information on $\supp |Du|$, so we cannot immediately conclude that if $u$ satisfies it, then it must be a minimizer of \eqref{eq:weightedROF}.

\subsection{The weighted string}
We can pose, for $F \in W^{1,2}_{1,1/\phi}(0,1)$ with $F(0)=0$, the minimization problem
\begin{equation}\label{eq:weightedstring}\min \left\{ \frac{1}{2} \int_0^1 \frac{1}{\phi(s)}\big(\tilde U'(s)\big)^2 \dd s \, \middle \vert \, \tilde U\in W^{1,2}_{1,1/\phi}(0,1), \, |\tilde U-F|\ls \alpha \rho, \, \tilde U(0)=0, \, \tilde U(1)=F(1)\right\},\end{equation}
which can be rewritten, with $U = \tilde U - F$, as
\begin{equation}\label{eq:weightedstring2}\min \left\{ \frac{1}{2} \int_0^1 \frac{1}{\phi(s)}\big( U'(s) + F'(s)\big)^2 \dd s \, \middle \vert \, U\in W^{1,2}_{1,1/\phi,0}(0,1), \, |U|\ls \alpha \rho \right\}.\end{equation}
\begin{lemma}
There is a unique minimizer $\tilde U_0$ of \eqref{eq:weightedstring}.
\end{lemma}
\begin{proof}
Since we have assumed that $\phi \in L^\infty(0,1) \cap C(0,1)$, we have the embedding \eqref{eq:continuousembedding} and the boundary values and constraints are well defined and closed. Moreover, the functional is nonnegative, convex and strongly continuous in $W^{1,2}_{1,1/\phi}(0,1)$, so weakly lower semicontinuous as well. Finally, since $U(0)=0$ and $U(1)=F(1)$ are fixed, the Poincaré inequality in $W^{1,2}_0(0,1)$ provides us with a bound for $\|U\|_{L^2(0,1)}$, so it is coercive in $W^{1,2}_{1,1/\phi}(0,1)$ and we may apply the direct method of the calculus of variations.
\end{proof}

\begin{theorem}\label{thm:weighteddual}
The Fenchel dual of \eqref{eq:weightedstring2}, for which strong duality holds, is equivalent to the weighted ROF problem \eqref{eq:weightedROF}. Specifically, if $U_0$ is the minimizer of \eqref{eq:weightedstring2} and $V_0$ is optimal in the dual problem, then $u_0 = V_0/\phi = U_0'/\phi + f$ is the minimizer of \eqref{eq:weightedROF}. Moreover, for the pair $(u_0, U_0)$ we have that \eqref{eq:stringoptimality} is satisfied (replacing $u,U$ by $u_0, U_0$, respectively), and this condition characterizes optimality of this duality pair.
\end{theorem}
\begin{proof}
Following for example \cite[Sec.~3.1]{HinPapRau17} and using the notation of \cite[Thm.~III.4.2]{EkeTem99}, we call
\[ \Lambda U = U' \quad \mathcal G(\Lambda U) = \frac12 \|\Lambda U+ \Lambda F\|_{L^2_{1/\phi}}^2 \quad \mathcal F = \chi_{\mathcal L},\]
where $\Lambda : W^{1,2}_{1,1/\phi,0}(0,1) \to L^2_{1/\phi}(0,1)$, $\chi_{\mathcal L}$ denotes the indicator function of $\mathcal L$ and 
\begin{equation}\label{eq:Lset}\mathcal L = \left\{ U \in W^{1,2}_{1,1/\phi,0}(0,1) \, \middle | \,|U| \ls \alpha \rho \right\},\end{equation}
so that \eqref{eq:weightedstring2} can be written as 
\[\min_{U \in W^{1,2}_{1,1/\phi,0}(0,1)} \mathcal F(U) + \mathcal G(\Lambda U).\]
In this situation, the dual problem writes
\[ \min_{V \in L^2_{1/\phi}(0,1)} \mathcal F^\ast(-\Lambda^\ast  V) + \mathcal G^\ast(V),\]
where
\[\mathcal G^\ast(V) = \int_0^1 \frac{1}{\phi} \left(\frac {V^2}{2} - \Lambda F V\right),\]
and taking into account $W^{1,2}_{1,1/\phi, 0}(0,1) \subset L^2(0,1) \subset \big(W^{1,2}_{1,1/\phi, 0}(0,1)\big)^\ast$ we have
\[\langle V, \Lambda U\rangle_{L^2_{1/\phi}} = \int_0^1 \frac{1}{\phi} V U' =  \scal{-\left(\frac{V}{\phi}\right)'}{U}_{\left(\big(W^{1,2}_{1,1/\phi,0}\big)^\ast, W^{1,2}_{1,1/\phi,0}\right)}\]
which in turn implies $\Lambda^\ast V = -(V/\phi)'$ and 
\begin{align*}\mathcal F^\ast(-\Lambda^\ast V) &= \sup_{U \in \mathcal L} \scal{\left(\frac{V}{\phi}\right)'}{U}_{\left(\big(W^{1,2}_{1,1/\phi,0}\big)^\ast, W^{1,2}_{1,1/\phi,0}\right)} \\&= \sup_{U \in \mathcal L} -\scal{\frac{V}{\phi}}{U'}_{L^2} = \sup_{U \in \mathcal L}\, \int_0^1 \frac{1}{\phi} U' V
\gs\alpha\TV_{\rho}(V / \phi),\end{align*}
where for the last inequality we have used that $\big\{U \in C^1_c(0,1) \,\big\vert\, |U|\ls \alpha \rho\big\} \subset \mathcal L$. We would like to have equality in this last inequality, which holds in particular when
\begin{equation}\label{eq:tightnessofdual}\alpha\TV_{\rho}(V / \phi) \gs \int_0^1 \frac{1}{\phi} U' V \text{ for every }V \in L^2_{1/\phi}(0,1) \cap \BV_\rho(0,1) \text{ and all }U \in \mathcal{L}.\end{equation}
Next, we notice that since \[(\TV_\rho)^\ast = \chi_{\overline{D(\mathcal L \cap C^1_c(0,1))}^{L ^2_{1/\phi}}},\]
where $DU=U'$ for $U \in W^{1,2}_{1,1/\phi,0}$, the statement \eqref{eq:tightnessofdual} is in fact equivalent to density of $C^1_c(0,1) \cap \mathcal{L}$ in $\mathcal{L}$ in the strong $W^{1,2}_{1,1/\phi}$ topology. Such a density property can not be obtained by directly mollifying elements of $\mathcal{L}$, because $\rho$ being nonconstant implies that the mollified functions could violate the pointwise constraint; a modified mollifying procedure has been considered in \cite{HinRau15}, but only for continuous $\rho$ with a positive lower bound, which does not cover our case. Instead, we can replace $V/\phi$ in \eqref{eq:tightnessofdual} by a sequence of smooth approximations for which $\TV_\rho$ converges (that is, in strict convergence), and then pass to the limit. This type of approximation of $\TV_\rho$ has been proved in \cite[Thm.~4.1.6]{Cam08} for Lipschitz $\rho$ but without the lower bound, allowing $\rho$ to vanish at the boundary and thus covering our situation.

Summarizing, the dual problem writes
\[\min_{V \in L^2_{1/\phi}(0,1)} \alpha \TV_\rho(V/\phi) + \int_0^1 \frac{1}{\phi}\left( \frac{V^2}{2} - F'V\right) \]
that clearly has the same minimizers as
\[\min_{V \in L^2_{1/\phi}(0,1)} \alpha \TV_\rho(V/\phi) + \int_0^1 \frac{1}{\phi} \frac{(V-F')^2}{2},\]
which in turn becomes \eqref{eq:weightedROF} if we define $u:=V/\phi$ and $f:=F'/\phi$, both of which belong to $L^2_\phi(0,1)$.

The optimality conditions for minimizers $U_0, V_0$ for this pair of problems are then
\begin{equation}\label{eq:dualstringoptimality}\begin{aligned}\Lambda U_0 &\in \partial_{L^2_{1/\phi}} \mathcal{G}^\ast(V_0) = \{V_0-\Lambda F\} = \{\phi(u_0 - f)\} \text{ for }u_0=\frac{V_0}{\phi}, \text{ and }\\
U_0 &\in \partial_{W^{1,2}_{1,1/\phi,0}} \mathcal{F}^\ast (-\Lambda^\ast V_0 ).\end{aligned}\end{equation}
Note that we use the results from \cite{EkeTem99}, which means that the subgradients of $\mathcal F^\ast$ are elements of the primal space $W^{1,2}_{1,1/\phi,0}(0,1)$ and not of the bidual $\big(W^{1,2}_{1,1/\phi,0}(0,1)\big)^{\ast\ast}$. In this case, we have (see \cite[Prop.~7]{BreHol16})
\[ \partial \mathcal F^\ast(0) = \overline{\mathcal L}^{W^{1,2}_{1,1/\phi}}, \]
where the closure is taken in the strong topology of $W^{1,2}_{1,1/\phi,0}(0,1)$. Moreover, we can also extend $\mathcal F$ (we denote by $\hat {\mathcal F}$ this extension) to $C_0([0,1])$, the space of continuous functions on $[0,1]$ which vanish on the boundary. Then, $\hat{\mathcal F}^\ast$ is defined on Radon measures and we have
\[\partial \hat{\mathcal F}^\ast(0) = \overline{\mathcal L}^{C_0}, \]
where the closure is now taken with respect to uniform convergence in $[0,1]$. In fact, these two closures satisfy
\begin{equation}\label{eq:phidensity}\overline{\mathcal L}^{W^{1,2}_{1,1/\phi}} \subset \overline{\mathcal L}^{C_0} = \big\{ U \in C_0([0,1]) \, \big | \,|U| \ls \alpha \rho \big\},\end{equation}
as shown in Lemma \ref{lem:phiclosure} below. This allows us to relate $\partial \mathcal{F}^\ast(0)$ to subgradients of the weighted total variation norm for measures, which (see for example \cite[Lem.~3.1]{HinPapRau17}) satisfy for each $\mu \in \big(C_0([0,1])\big)^\ast$ that
\[\partial \|\cdot\|_{\mathcal{M}_{\alpha \rho}}(\mu) \cap C_0([0,1]) = \left\{ U \in C_0([0,1]) \, \middle | \, U(x)=\alpha \rho(x)\frac{\mu}{|\mu|}(x) \text{ for }\mu\text{-a.e.~} x, \ |U| \ls \alpha \rho \right\},\]
where 
\[ \|\cdot\|_{\mathcal{M}_{\alpha \rho}}(\mu) = \sup\left\{ \int \varphi \dd \mu \ \middle\vert\ \varphi \in C_0([0,1]) \text{ with } |\varphi(x)| \ls \alpha \rho (x) \text{ for all } x \in [0,1] \right\}.\]
With this, in the two lines of \eqref{eq:dualstringoptimality} we have that 
\begin{gather*}
\Lambda U_0 = \phi(u_0-f)\text{ if and only if } 
U_0(x)=\int_0^x \phi(s)(u_0(s)-f(s)) \dd s \text{ for all }x\in(0,1), \text{ and }\\
\scal{-\Lambda^\ast V_0}{U_0}_{\left(\big(W^{1,2}_{1,1/\phi,0}\big)^\ast, W^{1,2}_{1,1/\phi,0}\right)} = \TV_\rho(V_0/\phi)\text{ if and only if } \\ U_0(x) = \alpha \rho(x) \frac{Du_0}{|Du_0|}(x) \text{ for }|Du_0|{-a.e.~}x \text{, and } |U_0(x)| \ls \alpha \rho(x) \text{ for all }x \in (0,1),\end{gather*}
with which we arrive at \eqref{eq:stringoptimality}.
\end{proof}

\begin{lemma}\label{lem:phiclosure}
Let $\mathcal L$ be defined as in \eqref{eq:Lset}. Then, \eqref{eq:phidensity} holds.
\end{lemma}
\begin{proof}
Since $W^{1,2}_{1,1/\phi, 0}(0,1)$ is continuously embedded in $C_0([0,1])$, we immediately have
\[\overline{\mathcal L}^{W^{1,2}_{1,1/\phi}} \subset \overline{\mathcal L}^{C_0}.\]
Moreover, the constraint $|U|\ls \alpha \rho$ is closed in $C_0([0,1])$, which shows
\[\overline{\mathcal L}^{C_0} \subset \big\{ U \in C_0([0,1]) \, \big | \,|U| \ls \alpha \rho \big\}.\]
For the opposite containment, we consider Pasch-Hausdorff regularizations with respect to the metric induced by the weight $\phi$ of the positive and negative parts $U^\pm$ of $U=U^+-U^-$, defined by
\[U_n^\pm(s) := \inf_{t \in (0,1)} \left\{ n\,d_\phi(t,s) + U^{\pm}(t) \right \},\]
where \[d_\phi(t,s) := \left | \int_t^s \phi(\omega) \dd \omega \right |.\]
Defining $U_n(s)$ by $U_n^+(s)$ if $U(s)>0$ and $-U_n^-(s)$ otherwise, so that again $U_n = U_n^+ - U_n^-$, we obtain regularized functions which are Lipschitz in this metric (this result dates back to \cite{Hau19}, see also \cite[Thm.~2.1]{Gut20} for a proof), hence belonging to $W^{1,\infty}_{1,1/\phi,0}(0,1) \subset W^{1,2}_{1,1/\phi,0}(0,1)$ and which are by definition pointwise bounded by $U$, so they remain in the constraint set. Moreover, the $U_n$ converge uniformly to $U$ in $[0,1]$, even if $\phi$ can vanish at the boundary. To see this, we notice that the functions $U_n^{\pm}$ are pointwise increasing with respect to $n$ and converging to $U^{\pm}$ respectively, to use Dini's theorem on both sequences.
\end{proof}

\subsection{Switching behavior}\label{sec:switching}
Defining the bijective transformation
\begin{align*}\Phi:[0,1]&\to \left[0, \int_0^1 \phi(\omega) \dd \omega\right]\\
s &\mapsto \int_0^{s} \phi(\omega) \dd \omega\end{align*}
and denoting $\check U := U \circ \Phi^{-1}$ for $U \in W^{1,2}_{1,1/\phi,0}(0,1)$, we have that $\check U' \in L^2(0,\Phi(1))$ and, using the Poincar\'e inequality in $W^{1,2}_0(0,\Phi(1))$, that $\check U \in L^2(0,\Phi(1))$ as well. Likewise, from $\check U \in W^{1,2}_0(0,\Phi(1))$ we can conclude $U \in W^{1,2}_{1,1/\phi,0}(0,1)$ using the inequality \[\|U\|_{L^2} \ls \|U\|_{L^\infty} \ls \|U'\|_{L^2} \ls \|U'\|_{L^2_{1/\phi}} \|\phi\|_{L^\infty}.\] This implies that \eqref{eq:weightedstring2} is equivalent to
\begin{equation}\label{eq:weightedstring3}\min \left\{ \frac{1}{2} \int_0^{\Phi(1)} \big(\check U'(t) + \check F'(t)\big)^2 \dd t \, \middle \vert \, \check U\in W^{1,2}_0\big(0,\Phi(1)\big), \, |\check U|\ls \alpha \left( \rho \circ \Phi^{-1} \right) \right\},\end{equation}
where similarly $\check F = F \circ \Phi^{-1}$. In turn, arguing as in \cite[Thm.~4.46]{SchGraGroHalLen09} the unique minimizer of \eqref{eq:weightedstring3} is the same for any strictly convex integrand, so in particular it can be obtained from
\begin{equation}\label{eq:weightedstring4}\min \left\{ \int_0^{\Phi(1)} \big[\big(\hat U'(t)\big)^2 + 1 \big]^\frac{1}{2} \dd t \, \middle \vert \, \hat U\in W^{1,2}_0\big(0,\Phi(1)\big), \, \check F - \alpha \left( \rho \circ \Phi^{-1} \right) \ls \hat U\ls \check F + \alpha \left( \rho \circ \Phi^{-1} \right) \right\},\end{equation}
by setting $\check U = \hat U - \check F$. This is now a `generalized' taut string formulation that fits in those considered in \cite[Lem.~5.4]{GraObe08}. One can then argue as in \cite[Lem.~9]{BauMunSieWar17} (which is directly based on the previously cited result) to conclude that for any $\delta >0$ the interval $(\delta, \Phi(1)-\delta)$ can be partitioned into finitely many subintervals on which one or neither of the constraints is active. The reasoning behind such a result is that in the form \eqref{eq:weightedstring4}, one is minimizing the Euclidean length of the graph of a continuous function with constraints from above and below, so as long as these constraints are at some positive distance $\epsilon$ apart, switching from one constraint to the other being active must cost at least $\epsilon$ and enforce a subinterval in which neither is active. In our case we need to restrict the interval to avoid the endpoints, because in our setting the weights are potentially degenerate and $\rho \circ \Phi^{-1}$ may vanish at $0$ or $\Phi(1)$, and with it the distance between the two constraints. For such degenerate weights, it is enough to assume that each of the two functions $\check{F}\pm \alpha \big(\rho \circ \Phi^{-1}\big)$ is either convex or concave on a neighborhood of $0$ and $1$ to ensure only finitely many changes of behavior.

With this property and taking into account Theorem \ref{thm:weighteddual} we can go back to \eqref{eq:stringoptimality} to see that the minimizer of \eqref{eq:weightedROF} alternates between the three behaviors
\begin{equation}\label{eq:weightedswitching}u(x) = f(x) - \alpha \frac{\rho'(x)}{\phi(x)},\quad u'(x)=0,\quad u(x) = f(x) + \alpha \frac{\rho'(x)}{\phi(x)}\quad\text{ for a.e.}~x\end{equation}
finitely many times on $(\delta,1-\delta)$ for any $\delta>0$, and on $(0,1)$ for data $f$ which is not oscillating near $0$ and $1$, and a subinterval corresponding to either the first or the last case is always followed by another on which $u'=0$.

\subsection{Denoising of unbounded radial data}
We can apply the results above to find minimizers in the case of radially symmetric data $f \in L^2(B(0,1))$ of 
\begin{equation}\label{eq:ballROF}\min_{u \in L^2(B(0,1))} \int_{B(0,1)} |u-f|^2 \dd x + \alpha \TV(u),\end{equation}
which considered in $\R^d$ corresponds to \eqref{eq:weightedROF} with $\rho(r)=\phi(r)=r^{d-1}$, 
\begin{equation}\label{eq:radialstringfns}\Phi(s)=\frac{1}{d} s^d,\quad(\rho \circ \Phi^{-1})(t) = (t d)^{(d-1)/d}, \ \text{ and}\quad \frac{\rho'(r)}{\phi(r)}=\frac{d-1}{r},\end{equation}
where we remark that \eqref{eq:weightedswitching} for this particular case was claimed without proof in \cite{Jal16}. Indeed, it is easy to guess the form of this term by formally differentiating \eqref{eq:stringoptimality}, but since this is not possible, we prefer to rigorously justify the behavior of minimizers by the arguments of the previous subsection. The earlier work \cite{StrCha96} also contains some results about nearly explicit minimizers for piecewise linear or piecewise constant radial data.

\begin{figure}
\centering
 \includegraphics[width = 0.45\textwidth]{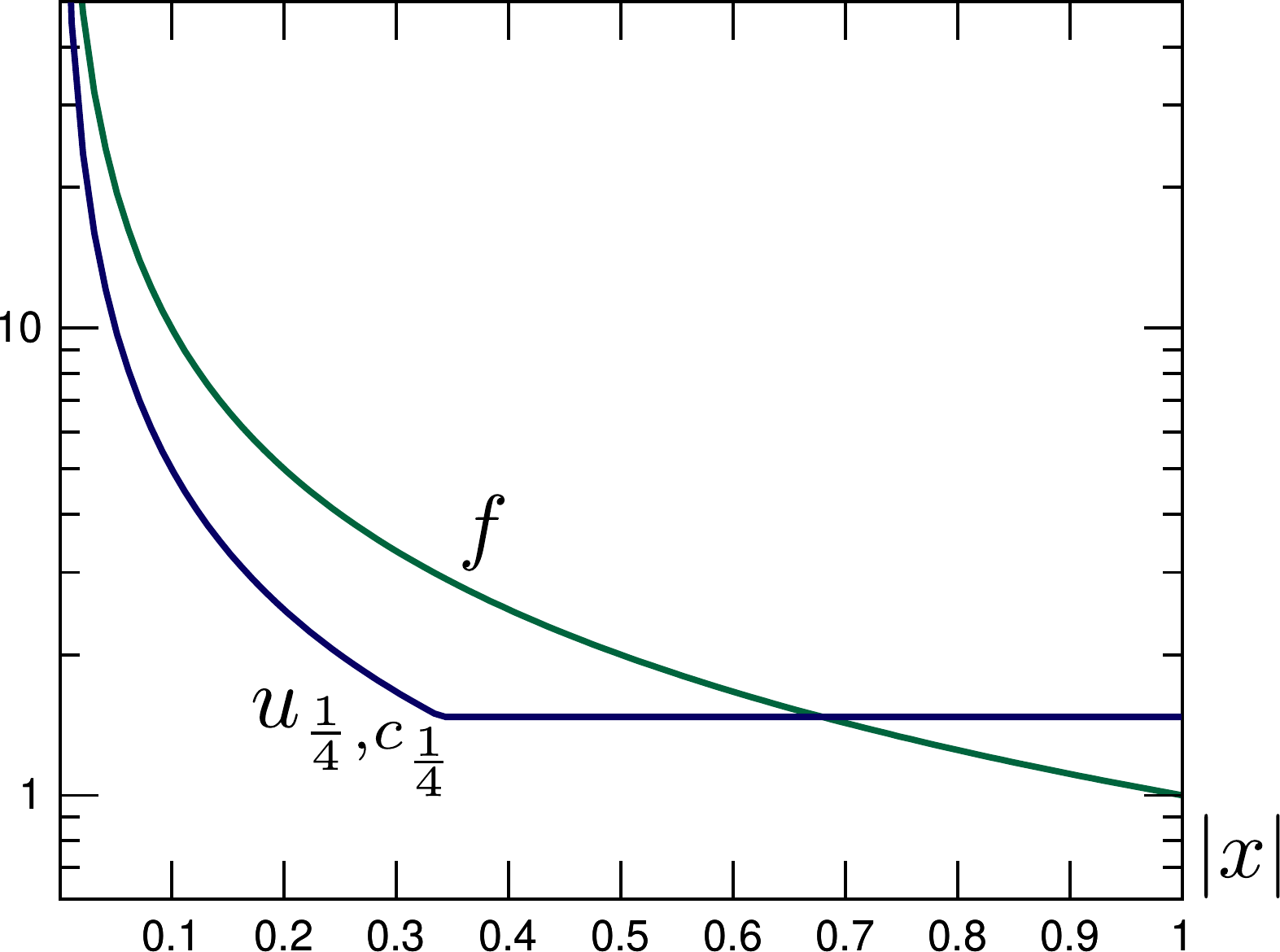}
 \caption{Logarithmic plot of the radial profile of the unbounded minimizer \eqref{eq:ualphac} of the denoising problem in $\R^3$ with data $f(x)=1/|x| \in L^2\big(B(0,1)\big)$, and $\alpha=1/4$ resulting in $c \approx 0.34$.}
 \label{fig:unbounded}
\end{figure}

In particular, let us compute explicitly the solution in the case $d=3$ and
\[f=\frac{1}{|\cdot|}\in L^2_{\loc}(\R^3).\]
In this case we have 
\[F(s)=\int_0^s f(r) \phi(r) \dd r = \frac{s^{d-1}}{d-1} = \frac{s^2}{2}\quad\text{and}\quad\check{F}(t) = (F \circ \Phi^{-1})(t) = \frac{(dt)^{(d-1)/d}}{d-1}=\frac{(3t)^{2/3}}{2},\]
and using \eqref{eq:radialstringfns} we conclude that $\check{F} + \alpha \big(\rho \circ \Phi^{-1}\big)$ is concave for all $\alpha >0$ while $\check{F} - \alpha \big(\rho \circ \Phi^{-1}\big)$ is strictly concave for $0< \alpha < 1/(d-1)=1/2$ and convex for $\alpha \geq 1/(d-1)=1/2$. Therefore, \eqref{eq:weightedswitching} tells us that either $u(r)=(1 - 2\alpha)/r$ or $u'(r)=0$ on each of finitely many subintervals. We consider the simplest case of two such intervals and check that we can satisfy the optimality condition for it. First, the Neumann boundary condition forces $u'(1)=0$ since the radial weight equals one there. Moreover, if $f \gs 0$, the minimizer $u$ of \eqref{eq:ballROF} satisfies $u\gs 0$ as well, because in this case, replacing $u$ with $u^+$ cannot increase either term of the energy. Taking this into account, as well as the fact that $u$ should be continuous \cite{CasChaNov11}, we consider for $c \in (0,1)$ the family of candidates
\begin{equation}\label{eq:ualphac}u_{\alpha, c}(r) := \begin{cases} \frac{1-2\alpha}{r} & \text{ for }r \in (0, c) \\ \frac{1-2\alpha}{c} & \text{ for }r \in (c, 1),\end{cases}\end{equation}
where $\alpha \ls 1/2$ is required so that $u \gs 0$ is satisfied. Moreover, we have also excluded segments with value $(1+2\alpha)/r$, since in that case the fidelity cost is the same as for $(1-2\alpha)/r$, but with a higher variation. To certify optimality of \eqref{eq:ualphac}, we attempt to satisfy the pointwise characterization \eqref{eq:stringoptimality}. The condition $U(1) = 0$, for $U$ defined by the first line of \eqref{eq:stringoptimality} forces
\[ \int_0^1 (u(r)-f(r)) r^{2} \dd r =  - \int_0^c \frac{2\alpha}{r} r^2 \dd r + \int_c^1 r^2 \left(\frac{1-2\alpha}{c}\right)  - r \dd r =0,\]
which is equivalent to
\[ \alpha c^2 + \frac{1-c^2}{2} + \frac{1-c^3}{3} \frac{2\alpha - 1}{c} = 0,\]
which simplifies to 
\begin{equation}\label{eq:cubic}c^3 + \frac{3c}{2\alpha -1} + 2 = 0,\end{equation}
which if $0 < \alpha < 1/2$ can be solved to find a $c \in (0,1)$, as can be easily seen using the intermediate value theorem. Now, since for the family $u_{\alpha,c}$, one has $\frac{Du}{|Du|} = -1$ on $(0,c] = \supp(|Du|)$, one can check that for $r$ on this interval, we have 
\[ U(r) = \alpha \rho(r) \frac{Du}{|Du|}(r). \]
Now, it remains to check that the inequality $|U|\ls \alpha \rho$ holds on $(c,1)$ (it holds on $(0,c)$ thanks to the previous argument). Since at $c$, both functions involved are equal, it is enough to show
\[ \sign\big(U(r)\big)(f(r) - u_{\alpha,c}(r)) \phi(r) =\frac{d}{dr} |U(r)| \ls \alpha \rho'(r) \text{ for } r>c,\]
which, using the definition of $u_{\alpha, c}$, is satisfied in particular if 
\[ \left \vert \frac 1r - \frac{1-2\alpha}{c} \right \vert r^2 \ls 2\alpha r, \text{ or } 1 \ls \frac{r}{c} \ls \frac{1+2\alpha}{1-2\alpha}\]
which, as long as $\frac 14 \ls \alpha < \frac 12$ and taking into account \eqref{eq:cubic}, is satisfied as soon as $c \ls r \ls 1$. Equation \eqref{eq:cubic} can be solved directly, for example for $\alpha = \frac14$, and we obtain a solution $c_{\frac14} \in (0,1)$, which guarantees that \eqref{eq:ualphac} is indeed a minimizer of \eqref{eq:ballROF}. Thus, we have obtained an explicit example of an unbounded minimizer of \eqref{eq:ballROF}, whose graph is depicted in Figure \ref{fig:unbounded}.
\begin{remark}\label{eq:radialzexplicit} We note here (as kindly pointed out by an anonymous reviewer) that using techniques of \cite{AndCasDiaMaz02} it is also possible to find an explicit vector field $\tilde z$ satisfying the characterization \eqref{eq:TVsubgrad2} to guarantee that \eqref{eq:ualphac} is a solution of the denoising problem with data $1/r$. Denoting $e_r = x/|x|$ and
\[\tilde z(x) = z(r)e_r, \text{ so that } \div z(x) = z'(r) + \frac{2}{r}z(r) \text{ for }d=3,\]
one can define
\[z(r)=\begin{cases}
-1 &\text{ for }r \in (0, c)\\
\frac{1}{2\alpha}\left( \frac{1}{r^2} - 1\right) + \frac{1-2\alpha}{3c\alpha}\left( r - \frac{1}{r^2}\right) &\text{ for }r \in [c, 1)\\
0 & \text{ for }r \in [1, +\infty)
\end{cases}\]
for which continuity as well as $L^\infty$ boundedness by $1$ holds for $c$ solving \eqref{eq:cubic}. Moreover
\[z'(r)+\frac{2}{r}z(r) = \begin{cases}-\frac{2}{r} &\text{ for }r \in (0, c)\\
\frac{1-2\alpha}{\alpha c} - \frac{1}{\alpha r} &\text{ for }r \in (c, 1),\end{cases}\]
so for such $c$ we have 
\[\div \tilde z \in \partial \TV(u_{\alpha, c})\text{ and }\frac{u_{\alpha,c}-f}{\alpha}=\div \tilde{z},\]
in which case $u_{\alpha, c}$ indeed solves \eqref{eq:ballROF}. Furthermore, a construction along these lines would also be possible for Lemma \ref{lem:boundconditions} below, for which our proof is based on the same weighted taut string technique.
\end{remark}

For other cases such an explicit computation may be impractical or impossible, but we can still use the taut string characterization to derive some conditions for local boundedness and unboundedness of radial minimizers:
\begin{lemma}\label{lem:boundconditions} Let $f:(0,1) \to \R$ be such that $\restr{f}{(\delta,1)} \in L^2(\delta, 1)$ for all $\delta \in (0,1)$, and such that for some $\delta_0 \in (0,1)$ we have
\[f(r) = \frac{1}{r^\beta} \text{ for all }r \in (0, \delta_0), \text{ so that }f \in L^2_{r^{d-1}}(0,1)\text{ for }2\beta < d,\]
and consider the corresponding minimizer $u$ of \eqref{eq:weightedROF} with $\rho(r)=\phi(r)=r^{d-1}$. Then, if $u$ is unbounded on any neighborhood of $0$, then necessarily $\beta \gs 1$, and if $\beta > 1$, then $u$ is unbounded on any neighborhood of $0$. In particular, for $d=2$ and with $f$ of this form belonging to $L^2_{r^{d-1}}(0,1)$, we have that $u$ is always bounded on a neighborhood of $0$.
\end{lemma}
\begin{proof}
First, as above we notice that $f \gs 0$ implies $u \gs 0$, since otherwise truncating $u$ to $u^+ = \max(u, 0)$ would produce a lower energy in \eqref{eq:weightedROF}. Then, the switching behavior \eqref{eq:weightedswitching} implies that there is $\delta \in (0, \delta_0)$ such that either 
\begin{equation}\label{eq:twooptions}u(r) = f(r) - \alpha\frac{d-1}{r},\text{ or }u'(r)= 0\text{ for all }r \in (0,\delta),\end{equation}
and one of the two choices must hold for the whole interval $(0,\delta)$ by the results cited in Section \ref{sec:switching}, since once again the behavior of $u$ can change only finitely many times away from $1$. To see this, notice that in this case
\[\begin{aligned}F(s)&=\int_0^s f(r) \phi(r) \dd r = \frac{s^{d-\beta}}{d-\beta} &&\text{for } s < \delta_0,\text{ and}\\\quad\check{F}(t) &= (F \circ \Phi^{-1})(t) = \frac{(dt)^{(d-\beta)/d}}{d-\beta} &&\text{for } t < (d\delta_0)^{1/d},\end{aligned}\]
so that $\check{F} + \alpha \big(\rho \circ \Phi^{-1}\big)$ is concave and $\check{F} - \alpha \big(\rho \circ \Phi^{-1}\big)$ is strictly concave around zero for $\beta > 1$ but convex for $\beta < 1$.

Now, we notice that if $\beta < 1$ we must have the second case of \eqref{eq:twooptions} because $u \gs 0$ would be violated in the first case, which proves the first part of the statement. Let us now assume that $\beta > 1$. Comparing in the energy \eqref{eq:weightedROF}  the first case of \eqref{eq:twooptions} holds exactly when for all $0<\nu \ls \delta$ we simultaneously have 
\begin{equation}\label{eq:nastyintegrals}\int_0^\nu \left[\left|f'(r)+\alpha\frac{d-1}{r^2}\right| + \alpha^2\left(\frac{d-1}{r}\right)^2 \right] r^{d-1} \dd r < \int_0^\nu \left(f(\nu) - \alpha\frac{d-1}{\nu} - f(r)\right)^2 r^{d-1} \dd r,\end{equation}
which encodes that switching to the second case $u'=0$ should not be advantageous at any $\nu \ls \delta$, and where equality is not possible since it would contradict uniqueness of minimizers of \eqref{eq:weightedROF}. Moreover, we also notice that the last term of the left hand integral can only be finite if $d \gs 3$, proving boundedness in case $d=2$. Using the specific form of $f(r)=r^{-\beta}$, inequality \eqref{eq:nastyintegrals} becomes
\begin{equation}\label{eq:nastyintegralsbeta}\int_0^\nu \left[\left|-\frac{\beta}{r^{\beta+1}}+\alpha\frac{d-1}{r^2}\right| + \alpha^2\left(\frac{d-1}{r}\right)^2 \right] r^{d-1} \dd r < \int_0^\nu \left(\frac{1}{\nu^\beta} - \alpha\frac{d-1}{\nu} - \frac{1}{r^\beta}\right)^2 r^{d-1} \dd r.\end{equation}
Using the notation $I(\nu) < J(\nu)$ in \eqref{eq:nastyintegralsbeta}, since the inequality is either true or false for every $\nu < \delta$ simultaneously, it holds if we had that
\[\ell := \lim_{\nu \to 0} \frac{I(\nu)}{J(\nu)} < 1.\]
But for $\nu$ small enough we have that $I(\nu)$ is dominated (up to a constant factor) by $\nu^{d-\beta-1}$ and $\nu^{d-2\beta}$ is dominated (also up to a factor) by $J(\nu)$, which, since $\beta >1$, implies $\ell = 0$.
\end{proof}
This lemma can be used to derive a boundedness criterion for more general radial data, without requiring local convexity or concavity assumptions for the functions appearing in the constraints of \eqref{eq:weightedstring4}:
\begin{prop}\label{prop:generalradial}
Assume that the radial data $f$ satisfies $f \gs 0$ and $f \in L^2_{r^{d-1}}(0,1)$. Then, if
\begin{equation}\label{eq:fastexplosion}\lim_{r \to 0} r^{\beta} f(r) = +\infty \text{ for some }\beta > 1,\end{equation}
then the corresponding minimizer $u$ of \eqref{eq:weightedROF} with $\rho(r)=\phi(r)=r^{d-1}$ is unbounded on any neighborhood of 0, but bounded on any compact interval excluding zero. Conversely, if 
\begin{equation}\label{eq:slowexplosion}\lim_{r \to 0} r^{\beta} f(r) = 0 \text{ for some }\beta < 1\end{equation}
then $u$ is bounded.
\end{prop}
\begin{proof}
First, we notice that we must have
\begin{equation}\label{eq:bddawayfromzero} u \in L^\infty(\delta, 1) \text{ for all } \delta > 0,\end{equation}
because on the interval $(\delta, 1)$ we have $\rho \geq \delta^{d-1}$, which ensures that the restriction of $u$ to it belongs to $\BV(\delta, 1)$, and in one dimension we have the embedding $\BV(\delta, 1) \subset L^\infty(\delta,1)$. Given \eqref{eq:bddawayfromzero} and using the equivalence with \eqref{eq:ballROF}, we can conclude by applying the pointwise comparison principle for minimizers of $\TV$ denoising in $B(0,1) \subset \R^d$ (a proof of which can be found in for example \cite[Prop.~3.6]{IglMer21}) against inputs of the form treated in Lemma \ref{lem:boundconditions} above. Assuming \eqref{eq:fastexplosion}, we choose $\delta_0$ such that $f(r) \geq r^{-\beta}$ for $r \in (0, \delta_0)$, while for the case \eqref{eq:slowexplosion} we take $\delta_0$ such that $f(r) \leq r^{-\beta}$ for $r \in (0, \delta_0)$, and in both cases compare with $\tilde{f}$ defined by
\[\tilde{f}(r) = \begin{cases}\frac{1}{r^\beta} &\text{ if }r \in (0, \delta_0) \\ f(r) &\text{ if }r \in [\delta_0,1).\end{cases}\qedhere\]
\end{proof}
We can also see this criterion in light of the results of Section \ref{sec:boundedness}, in particular Propositions \ref{prop:single_u_bounded} and \ref{prop:bounded_domains}. There, the main requirement for boundedness of $u$ is that there should be $v \in \partial_{L^{d/(d-1)}} \TV(u) \subset L^d(\Omega)$. In the case of radial data on $\Omega = B(0,1) \subset \R^d$, we have worked with the ROF denoising problem for which
\[v_\alpha := -\frac{1}{\alpha}(u_\alpha-f) \in \partial_{L^2} \TV(u_\alpha),\]
and by \eqref{eq:weightedswitching} we know that $u_\alpha-f$ switches between $\pm \alpha (d-1)/|\cdot|$ and zero on annuli. Moreover, the power growth $1/|\cdot|$ is precisely the threshold for a function on $\R^d$ to belong to $L^d(B(0,1))$, assuming it is in $L^\infty \big(B(0,1)\setminus B(0,\delta)\big)$ for all $\delta$. That is, for slower power growth as in \eqref{eq:slowexplosion} we have $v_{\alpha} \in L^d(B(0,1))$ and the minimizer is bounded, which we could have proved using the techniques of Section 3. In contrast, for faster power growth as in \eqref{eq:fastexplosion} we have that $v_{\alpha} \notin L^d(B(0,1))$, so the methods of Section \ref{sec:boundedness} are not applicable, and by Proposition \ref{prop:generalradial} in this case $\TV$ denoising produces an unbounded minimizer.

\section{Higher-order regularization terms}\label{sec:higherorder}
Here we treat problems regularized with two popular approaches combining derivatives of first and higher orders, with the goal of extending our boundedness results from Section \ref{sec:boundedness} to these more involved settings. In this section, we limit ourselves to the case of functions defined on a bounded Lipschitz domain $\Omega \subset \R^d$.

The first approach, for which we are able to prove an analog of Theorem \ref{thm:joint_thm} in dimension $2$, is infimal convolution of first and second order total variation first introduced in \cite{ChaLio97} (see also \cite{BreHol20}), that is
\begin{equation}\min_{ u \in L^{d/(d-1)}(\Omega)} \frac{1}{\sigma}\left( \int_\Sigma |Au - (f+w)|^q \right)^{\sigma/q} + \left[\,\inf_{g \in \BV^2(\Omega)}\alpha_1 \TV(u-g) + \alpha_2 \TV^2(g)\right]\label{eq:primal-ictv},\end{equation}
where $\TV^2(g)$ for $g \in \BV^2(\Omega) \subset W^{1,1}(\Omega)$ is the norm of the distributional Hessian as a matrix-valued Radon measure, that is
\begin{equation}\label{eq:TV2def}\TV^2(g) := \sup\left\{ \int_\Omega \nabla g(x) \cdot \div M(x) \dd x \, \middle\vert\, M \in C^1_c(\Omega; \R^{d \times d}) \text{ with } |M(x)|_F \ls 1 \text{ for all } x\right\},\end{equation}
with $|M|_F$ denoting the Frobenius norm, that is $|M|^2_F=\sum_{ij} (M_{ij})^2$. Here and in the rest of this section, $A$ and the exponents $\sigma,q$ are as in Sections \ref{sec:intro} and \ref{sec:boundedness}, i.e. $A: L^{d/(d-1)}(\Omega) \to L^q(\Sigma)$ is a bounded linear operator, and $\sigma = \min(q,2)$.

The second approach, widely used in applications but harder to treat analytically, is total generalized variation of second order, introduced in \cite{BreKunPoc10} and for which we use the characterization of \cite[Thm.~3.1]{BreVal11} to write
\begin{equation}\min_{ u \in L^{d/(d-1)}(\Omega)} \frac{1}{\sigma}\left( \int_\Sigma |Au - (f+w)|^q \right)^{\sigma/q} + \left[\inf_{z \in \BD(\Omega)}\alpha_1 |Du - z|(\Omega) + \alpha_2 \TD(z)\right]\label{eq:primal-tgv},\end{equation}
for which we denote the second term as $\TGV(u)$, not making the regularization parameters $\alpha_1, \alpha_2$ explicit in the notation. Here, $\TD(z)$ denotes the `total deformation' of the vector field $z$ defined in terms of its distributional symmetric gradient, that is 
\begin{equation}\label{eq:TDdef}\TD(z) := \sup\left\{ \int_\Omega z(x) \cdot \div M(x) \dd x \, \middle\vert\, M \in C^1_c(\Omega; \R_{\text{sym}}^{d \times d}) \text{ with } |M(x)|_F \ls 1 \text{ for all } x\right\}.\end{equation}

Since we are working on a bounded Lipschitz domain, the inner infima are attained in both cases and the denoising functionals \eqref{eq:primal-ictv} and \eqref{eq:primal-tgv} have unique minimizers. For a proof, see \cite[Prop.~4.8, Prop.~4.10]{BreHol20} for \eqref{eq:primal-ictv} and \cite[Thm.~3.1, Thm.~4.2]{BreVal11}\cite{BreHol14}\cite[Thm.~5.9, Prop.~5.17]{BreHol20} for the TGV case \eqref{eq:primal-tgv}.
\begin{lemma}\label{lem:higherorderoptimality}Assuming that $u$ is a minimizer of \eqref{eq:primal-ictv} and the inner infimum is attained by some $g_u$, we have the necessary and sufficient optimality condition:
\begin{equation}\begin{gathered}v \in \alpha_1 \partial\TV(u-g_u) \cap \alpha_2 \partial\TV^2(g_u), \text{ for}\\
v:=\|Au-f-w\|_{L^q(\Sigma)}^{\sigma-2} \,A^\ast\left( |f+w-Au|^{q-2}(f+w-Au) \right)
\label{eq:opt-ictv}\end{gathered}\end{equation}
If \eqref{eq:primal-tgv} is minimized by $u$ and the inner infimum is attained for some $z_u$, we also have the necessary condition
\begin{equation}v \in \alpha_1 \partial J_{z_u}(u)\label{eq:opt-tgv},\end{equation}
where the functional $J_{z_u}$ is defined by $J_{z_u}(\tilde{u}) = |D\tilde{u} - z_u|(\Omega)$, and $v$ is as in \eqref{eq:opt-ictv}.
\end{lemma}
\begin{proof}
We assume for the sake of simplicity that $\alpha_1=\alpha_2 = 1$, and first prove \eqref{eq:opt-tgv}. We notice that the first term in \eqref{eq:primal-tgv} is continuous with respect to $u \in L^{d/(d-1)}(\Omega)$, so we have \cite[Prop.~I.5.6]{EkeTem99} for an optimal $u$ that 
\begin{equation}v \in \partial \left[\inf_{z \in \BD(\Omega)} |D\cdot - z|(\Omega) + \TD(z)\right](u)\label{eq:preopt-tgv},\end{equation}
and similarly for \eqref{eq:primal-ictv}. Now, one would like to use the characterization of the subgradient of exact infimal convolutions as intersections of subgradients at the minimizing pair (see \cite[Cor.~2.4.7]{Zal02}, for example). Since the complete result cannot be applied directly to \eqref{eq:preopt-tgv}, we make explicit the parts of its proof that are applicable. To this end, let $u^\ast \in L^d(\Omega)$ belong to the right hand side of \eqref{eq:preopt-tgv}. We have for all $\tilde{u} \in L^{d/(d-1)}(\Omega)$, noticing that $\TGV(u)<+\infty$ implies $u \in \BV(\Omega)$ since we may just consider $z=0$ in the inner minimization, that
\begin{align*}&|D(u+\tilde{u}) - z_u|(\Omega)- |Du - z_u|(\Omega) -\scal{u^\ast}{\tilde{u}}_{L^d \times L^{d/(d-1)}} \\
&\quad =|D(u+\tilde{u}) - z_u|(\Omega) + \TD(z_u) - |Du - z_u|(\Omega) - \TD(z_u) -\scal{u^\ast}{\tilde{u}}_{L^d \times L^{d/(d-1)}} \\
&\quad =|D(u+\tilde{u}) - z_u|(\Omega) + \TD(z_u) - \left[\inf_{z \in \BD(\Omega)} |Du - z|(\Omega) + \TD(z)\right] -\scal{u^\ast}{\tilde{u}}_{L^d \times L^{d/(d-1)}} \\
&\quad \gs \left[\inf_{z \in \BD(\Omega)} |D(u+\tilde{u}) - z|(\Omega) + \TD(z)\right]\! - \!\left[\inf_{z \in \BD(\Omega)} |Du - z|(\Omega) + \TD(z)\right] \! - \! \scal{u^\ast}{\tilde{u}}_{L^d \times L^{d/(d-1)}} \\&\quad\gs 0,
\end{align*}
which combined with \eqref{eq:preopt-tgv} implies \eqref{eq:opt-tgv}. To prove that \eqref{eq:opt-ictv} is necessary for optimality, assuming that $u$ now minimizes \eqref{eq:primal-ictv} we have analogously
\begin{equation}v \in \partial \left[\inf_{\,g \in L^2(\Omega)} \TV(\cdot-g) + \TV^2(g)\right](u)\label{eq:preopt-ictv},\end{equation} 
and the analogous computation leads to $v \in \partial\TV(u-g_u)$. To see that we also have $v \in \partial\TV^2(g_u)$, we consider $u^\ast \in L^d(\Omega)$ in the right hand side of \eqref{eq:preopt-ictv} so that for $\tilde u \in L^{d/(d-1)}(\Omega)$,
\begin{align*}&\TV^2(g_u+\tilde{u}) - \TV^2(g_u) - \scal{u^\ast}{\tilde{u}}_{L^2}\\
&\quad=\TV^2(g_u+\tilde{u}) + \TV(u-g_u) - \TV^2(g_u) - \TV(u-g_u) -\scal{u^\ast}{\tilde{u}}_{L^d \times L^{d/(d-1)}}\\
&\quad=\TV^2(g_u+\tilde{u}) + \TV(u-g_u) - \left[\,\inf_{g \in L^{d/(d-1)}(\Omega)} \TV(u-g) + \TV^2(g) \right] -\scal{u^\ast}{\tilde{u}}_{L^d \times L^{d/(d-1)}} \\
&\quad\gs \left[\,\inf_{g \in L^{d/(d-1)}(\Omega)} \TV(u+\tilde{u}-g) + \TV^2(g)\right] \\
&\quad\quad\quad\quad - \left[\,\inf_{g \in L^{d/(d-1)}(\Omega)} \TV(u-g) + \TV^2(g) \right] -\scal{u^\ast}{\tilde{u}}_{L^d \times L^{d/(d-1)}} \gs 0.
\end{align*}
Notice that this last part is not a priori possible for \eqref{eq:primal-tgv} due to the lack of symmetry in the variables involved. That \eqref{eq:opt-ictv} is also sufficient can be seen from \cite[Cor.~2.4.7]{Zal02}, since in this case we do have an infimal convolution.
\end{proof}

\begin{prop}\label{prop:ictv}Whenever $d=2$, minimizers of \eqref{eq:primal-ictv} are in $L^\infty(\Omega)$. For $d>2$ they belong to $L^{d/(d-2)}(\Omega)$.
\end{prop}
\begin{proof}
Using \eqref{eq:opt-ictv} we notice as in Remark \ref{rem:nonemptysubg} that the proof of Proposition \ref{prop:single_u_bounded} implies that $u-g_u \in L^\infty(\Omega)$. Now since $\TV^2(g_u) < +\infty$, for $d>2$ one gets $g_u \in L^{d/(d-2)}(\Omega)$ so that also $u\in L^{d/(d-2)}(\Omega)$, and in fact for $d=2$ the critical embedding $\BV^2(\Omega) \subset L^\infty(\Omega)$ holds, see \cite[Thm.~3.2]{Dem84} for domains with $C^2$ boundary and \cite[Rem.~0.3]{Sav96} for Lipschitz domains. Therefore, in that case $u \in L^\infty(\Omega)$.
\end{proof}

More generally one can prove the following restricted analogue of Theorem \ref{thm:joint_thm}:
\begin{theorem}\label{thm:joint_thm_ictv}
Let $d=2$ and $\Omega$ be a bounded Lipschitz domain. Assume that for $A:L^2(\Omega)\to L^q(\Sigma)$ linear bounded, $f \in L^q(\Sigma)$ and a fixed $\gamma >0$ there is a unique solution $u^\dag$ for $Au=f$ which satisfies the source condition $\operatorname{Ran}(A^\ast) \cap \partial (\TV \square \, \gamma\! \TV^2)(u^\dag) \neq \emptyset$, where we denote
\begin{equation}\label{eq:ictv-again}(\TV \square \, \gamma\! \TV^2)(u):=\inf_{g \in \BV^2(\Omega)} \TV(u-g) + \gamma \TV^2(g).\end{equation}
In this case, there is some constant $C(\gamma, q,\sigma, \Omega)$ such that if $\alpha_n, w_n$ are sequences of regularization parameters and perturbations for which
\[\alpha_n \gs C(\gamma, q,\sigma, \Omega) \|A^\ast\| \Vert w_n \Vert_{L^q(\Sigma)}^{\sigma-1} \xrightarrow[n \to \infty]{} 0\]
then the corresponding sequence $u_n$ of solutions of 
\begin{equation}\label{eq:primalproblem-ictv}\min_{ u \in L^2(\Omega)} \frac{1}{\sigma}\left( \int_\Sigma |Au - (f+w_n)|^q \right)^{\sigma/q}+ \alpha_n (\TV \square \, \gamma\! \TV^2)(u)\end{equation}
is bounded in $L^\infty(\Omega)$, and (possibly up to a subsequence)
\[u_n \xrightarrow[n \to \infty]{} u^\dag\text{ strongly in }L^{\overline p}(\Omega)\text{ for all }\overline p \in (1, \infty).\]
\end{theorem}

\begin{proof}
Since the proof follows by using the methods of Section \ref{sec:boundedness} essentially verbatim, we provide just an outline highlighting the steps that pose significant differences.

The first ingredient is to repeat the results on convergence and stability of dual variables summarized in Section \ref{sec:dualstuff}, which in fact only depend on the regularization term being positively one-homogeneous and lower semicontinuous. Once these are obtained, from the characterization \eqref{eq:opt-ictv} and Proposition \ref{prop:single_u_bounded} one gets uniform $L^\infty$ bounds for $u_n - g_{u_n}$. To finish, one notices that by comparing with $u^\dag$ in the minimization problem \eqref{eq:primalproblem-ictv} we get using $A u^\dag = f$ that
\[\gamma\TV^2(g_{u_n}) \ls \frac{1}{\sigma} \|w_n\|_{L^q(\Sigma)}^\sigma + \alpha_n \left[\,\inf_{g \in \BV^2(\Omega)} \TV(u^\dag-g) + \gamma \TV^2(g)\right],\]
in which the first term of the right hand side is bounded above in terms of $\alpha_n$ by the parameter choice and $(\TV \square \, \gamma\! \TV^2)(u^\dag) < +\infty$ by the source condition, so the bound on $g_{u_n}$ obtained by the embedding $\BV^2(\Omega) \subset L^\infty(\Omega)$ is not just uniform in $n$ but in fact vanishes as $n \to \infty$. Since we have assumed that $\Omega$ is bounded, the strong convergence in $L^{\bar{p}}(\Omega)$ follows directly along the lines of the proof of Corollary \ref{cor:strong_conv}.
\end{proof}
In comparison with Theorem \ref{thm:joint_thm}, the above result is weaker in two aspects. The first is that  we need to impose that the two regularization parameters maintain a constant ratio between them, in order to formulate a source condition in terms of a subgradient of a fixed functional. Moreover, the result is limited to $d=2$ and bounded domains. Without the latter assumption, in particular the upgrade to strong convergence in $L^{\bar{p}}(\Omega)$ would require a common compact support for all $u_n$. However, even assuming attainment of the inner infimum, using Lemma \ref{lem:higherorderoptimality} we would only get a common support for $u_n-g_{u_n}$, and in the absence of coarea formula for $\TV^2$ it is not clear how to control the one of $g_{u_n}$.

For the case of $\TGV$ regularization it is unclear if one can also use the boundedness results of Section \ref{sec:boundedness}, which are based on subgradients of $\TV$. A first observation is that we cannot use the $\TGV$ subgradients directly, except in trivial cases:

\begin{prop}\label{prop:uninterestingsubgrad}If for $u \in L^{d/(d-1)}(\Omega)$ and $v \in \partial \TGV(u)$ we have that also $v \in \alpha_1 \partial\TV(u)$, then necessarily $\alpha_1 \TV(u)=\TGV(u)$.
\end{prop}
\begin{proof}
Assuming without loss of generality that $\alpha_1 = 1$, we have by testing with $z=0$ that
\[\inf_{z \in \BD(\Omega)} |Du - z|(\Omega) + \alpha_2 \TD(z) \ls |Du|(\Omega) = \TV(u)  \text{ for all }u\in L^{d/(d-1)}(\Omega),\]
which since the Fenchel conjugate is order reversing means that
\[\TV^\ast (v) \ls \TGV^\ast (v) \text{ for all }v\in L^d(\Omega).\]
Because both these functionals are convex positively one-homogeneous we have as in \eqref{eq:tvast} that
\[\TV^\ast = \chi_{\partial \TV (0)}\text{ and }\TGV^\ast = \chi_{\partial \TGV (0)},\]
so that 
\begin{equation}\label{eq:zeroinclusion}\chi_{\partial \TV (0)} \ls \chi_{\partial \TGV (0)}, \text{ or } \partial\TGV(0) \subset \partial\TV(0).\end{equation}
Now, again since these functionals are positively one-homogeneous we have
\begin{equation}\label{eq:TGVsubgrad1}v \in \partial \TGV (u)\text{ exactly when }v \in \partial \TGV(0) \text{ and }\int_\Omega vu = \TGV(u)\end{equation}
and similarly for $\TV$, as stated in \eqref{eq:TVsubgrad1}. But since we assumed that $v \in \partial\TV(u)$ as well, we must then have
\[\TGV(u) = \int_\Omega vu = \TV(u).\qedhere\]
\end{proof}
In Appendix \ref{sec:counterexample} we explore whether it is possible to use a more refined approach by trying to find elements of $\partial \TV(u)$ from those of $\partial J_{z_u}(u)$ appearing in \eqref{eq:opt-tgv} for $\TGV$ regularization, with a negative conclusion (at least without using further properties of $z_u$) in the form of an explicit counterexample.

\section*{Acknowledgments} We wish to thank Martin Holler for some enlightening discussions about higher-order regularization terms in the early stages of this work. Moreover, we also would like to thank the anonymous reviewers for their careful reading of this manuscript, and for the explicit construction in Remark \ref{eq:radialzexplicit} provided by one of them.

\section*{Declarations} A large portion of this work was completed while the second-named author was employed at the Johann Radon Institute for Computational and Applied Mathematics (RICAM) of the Austrian Academy of Sciences (\"OAW), during which his work was partially supported by the State of Upper Austria. The Institute of Mathematics and Scientific Computing of the University of Graz, with which the first-named author is affiliated, is a member of NAWI Graz (\texttt{https://www.nawigraz.at/en/}). 

The authors have no conflicts of interest to declare that are relevant to the content of this article.

\bibliographystyle{plain}
\bibliography{biblio}

\appendix
\section{A subgradient of TGV does not easily induce one for TV}\label{sec:counterexample}

To attempt to infer that $\TGV$-regularized minimizers are bounded, an alternative approach to the arguments of Section \ref{sec:higherorder} would be to try and use that the inner minimization problem is attained. Proposition \ref{prop:ictv} above for the infimal convolution case uses only that $\partial\TV(u-g_u) \neq \emptyset$ and $g_u \in \BV^2(\Omega)$, and not any information about subdifferentials of $\TV^2$. It is natural to wonder whether it is possible to follow a similar approach for $\TGV$ starting from \eqref{eq:opt-tgv}. Unlike in the infimal convolution case, clearly this cannot follow just from embeddings. 

What \eqref{eq:opt-tgv} provides is an element in the subdifferential of the functional $J_z=|D\cdot\,-z|(\Omega)$ at $u$, and $Du$ is only a measure but $z \in \BD(\Omega) \subset L^{d/(d-1)}(\Omega; \R^d)$ is much more regular. We could then ask ourselves if this means that $\partial \TV(u) \neq \emptyset$ or, equivalently, if $u$ being irregular enough to not have a subgradient of $\TV$ would imply $\partial J_z(u)=\emptyset$ as well, making $z$ just a `more regular perturbation'. Unfortunately this is not the case, which prevents using \eqref{eq:opt-tgv} and knowledge of subgradients of $\TV$ to infer whether minimizers of $\TGV$ regularization are bounded.

For our counterexample we will use the following characterization:

\begin{lemma}\label{lem:badseqs}For arbitrary $u \in \BV(\Omega)$ and $z \in \mathcal{M}(\Omega; \R^d)$ (which could be $z=0$ for the $\TV$ case) we have that $\partial J_{z}(u) = \emptyset$ if and only if there are $v_n$ with $\|v_n\|_{L^{d/(d-1)}(\Omega)}=1$ and $t_n \to 0^+$ such that
\begin{equation}\label{eq:explosion}\lim_{n \to \infty} \frac{1}{t_n}\Big(|Du + t_n Dv_n - z|(\Omega) - |Du - z|(\Omega) \Big) = -\infty.\end{equation}
\end{lemma}
\begin{proof}
This is a general fact for proper convex functionals whose domain is a linear subspace. Let $F$ be such a functional on a separable Banach space $X$ and let $x \in \dom F$. For any $h \in X$, the directional derivative \[D_h F(x) := \lim_{t \searrow 0}\frac{F(x+th) - F(x)}{t}\] in the direction $h$ exists (it can be $\pm\infty$) since the difference quotient
 \[ t \mapsto \frac{F(x+th) - F(x)}{t} \]
is either constant $+\infty$ or nondecreasing with $t>0$. Moreover, we notice that $D_{(\cdot)} F(x)$ is sublinear and positively homogeneous, in particular convex. And since the domain of $F$ is a linear subspace, it is its own relative interior, so $D_{(\cdot)} F(x)$ is also proper \cite[Thm.~2.1.13]{Zal02}.
 
Let us assume that there is $v \in X^\ast$ such that $v \in \partial F(x).$ Then, we have $\frac{F(x+th) - F(x)}{t} \gs \scal{v}{h}$, which implies $D_h F(x) \gs \scal{v}{h}.$ If in particular, $\Vert h \Vert = 1$, we have, for all $t$, 
 \[\frac{F(x+th) - F(x)}{t} \gs - \Vert v \Vert.\] In consequence, if $\partial F(x) \neq \emptyset$ the limit in \eqref{eq:explosion} cannot be $-\infty$.
 
For the opposite direction, let us now assume that $\partial F(x) = \emptyset$ for $x \in \dom F$. We want to show that in this case
\[\big(D_{(\cdot)} F(x)\big)^\ast = +\infty.\]
For that, let us assume the contrary, that is, at some $v \in X^\ast$,
\[ \big(D_{(\cdot)} F(x)\big)^\ast (v) = \sup_{h \in X} \scal{v}{h} - D_hF(x) =M <  + \infty.\]
Noticing that since $D_{(\cdot)} F(x)$ is positively one-homogeneous, $\big(D_{(\cdot)} F(x)\big)^\ast$ is a characteristic function, so we must have that $M=0$. This means that for all $h$, we have 
\[D_h F(x) \gs \scal{v}{h},\]
which implies that for all $t>0$,
\[\frac{F(x+th) - F(x)}{t} \gs \scal{v}{h}\]
or again, using $t=1$
\[ F(x+h) - F(x) \gs \scal{v}{h}  \]
and $v \in \partial F(x)$, which is a contradiction.
 
Conjugating again, we get $\big(D_{(\cdot)} F(x)\big)^{\ast \ast} = -\infty$. This implies that the lower semicontinuous envelope of the convex function $D_{(\cdot)} F(x)$ is not proper, since otherwise \cite[Thm.~2.3.4(i)]{Zal02} it would equal $\big(D_{(\cdot)} F(x)\big)^{\ast \ast}$, giving a contradiction. And since $D_{(\cdot)} F(x)$ is not everywhere $+\infty$ (just consider $D_0 F(x)=0$), there must be some $h \in X$ at which this lower semicontinuous envelope takes the value $-\infty$. This means that there is a sequence of directions $h_n \to h$ for which $D_{h_n} F(x) \to -\infty$, and by removing some elements if necessary, we may also assume that $D_{h_n} F(x) < +\infty$ and $h_n \neq 0$ for all $n$ (again since $D_0 F(x)=0$).

 By definition of the directional derivative, there exists $t_n$ (we can choose $t_{n+1} < t_n/2$ to enforce $t_n \to 0$) such that 
 \[ \left \vert \frac{F(x+t_nh_n) - F(x)}{t_n} - D_{h_n} F(x) \right \vert \ls 1,\]
 where the left hand is well defined since $D_{h_n} F(x) < +\infty$.
 
 We obtain then
 \[\frac{F(x+t_nh_n) - F(x)}{t_n} \to -\infty.\]
 Finally, setting $\tilde h_n = \frac{h_n}{\Vert h_n \Vert}$ and $\tilde t_n = \Vert h_n \Vert t_n$, we have $\tilde t_n \to 0$ and
  \[\frac{F(x+\tilde t_n \tilde h_n) - F(x)}{\tilde t_n} \to -\infty,\]
which was to be shown.
 \end{proof}
 
\begin{example}\label{ex:nosubdiff}
We start with some nonzero function $S \in W^{1,1}(0,1)$ with $S(0)=S(1)=0$. For illustration purposes one can think for example of the hat function
\[S(x)=1-|2x-1|.\]
We extend it by zero outside of $(0,1)$ and define the functions, for $n\gs 0$,
\[ S_n(x) := S\left[2^n x-\left(2^{n+1}-2\right) \right] \]
which have support in \[\left[2-2^{-n+1}, 2-2^{-n}\right] \subset [0,2].\]
Note that these supports have an intersection of measure zero. 
We then introduce $\Omega = (0,2)\times(0,1)$ and define $u_n \in W^{1,1}(\Omega)$ by
\[u_n(x_1,x_2):=S_n(x_1).\]
Now, we have that
\begin{align}\|\nabla u_n\|_{L^1} &= \int_{\Omega} |\nabla u_n(x_1,x_2)| \dd x = \int_{(0,2)} |S_n'(t)| \dd t \\&=  \int_{(0,1)} |S'(t)| \dd t = \|S'\|_{L^1(0,1)},\end{align}
while
\[\|u_n\|_{L^2}^2=\int_{\Omega} |u_n(x_1,x_2)|^2 \dd x = \frac{1}{2^n} \int_{(0,1)} |S(t)|^2 \dd t = \frac{1}{2^n} \|S\|_{L^2(0,1)}^2.\]
We define then
\[u:=\sum_{n=0}^\infty 2^{-2n} u_n \in W^{1,1}(\Omega).\]
For it we can use Lemma \ref{lem:badseqs} with $z=0$,
\[v_n:=-\frac{u_n}{\|u_n\|_{L^2}},\quad t_n:=2^{-2n}\|u_n\|_{L^2}\to 0, \quad|D v_n|(\Omega) = 2^{n/2} \frac{\|S'\|_{L^1(0,1)}}{\|S\|_{L^2(0,1)}} \to \infty.\]
Moreover by construction $|Du+t_n Dv_n|=|Du|-t_n |Dv_n|$, so we conclude that $|D \cdot|\big((0,2)\times (0,1)\big)$ is not subdifferentiable at $u$. In particular, for the particular choice $z=Du$ we have $0 \in \partial J_z(u)$, but $\partial \TV(u) = \emptyset$.
\begin{figure}
\centering
 \includegraphics[width = 0.45\textwidth]{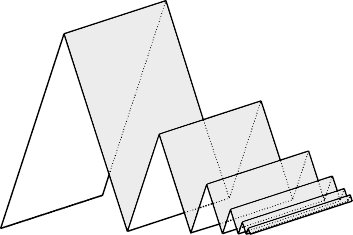}
 \caption{Sketch of the graph of the function $u$ in Example \ref{ex:nosubdiff}.}
 \label{fig:counterexample}
\end{figure}
Notice that $Du \in L^1(\Omega; \R^2)$, indicating that lack of subdifferentiability of $\TV$ arises not only from the singular part of the derivative of $\BV$ functions. We also observe that, extending the $v_n$ to $\R^2$, we have
\[\frac{|D v_n|}{|Dv_n|(\Omega)} \wksto \mathcal{H}^1 \mres \big(\{2\}\times (0,1)\big),\]
and in particular the Lebesgue measure of the support of $D v_n$ vanishes in the limit.
\end{example}

\begin{remark}From the example above we also get a counterexample of the statement we wanted to prove in the first place. The idea is just to take $u$ as above, and for $z$ consisting of the tail of the series starting from some $N>0$ to remove the small amplitude oscillations in $Du-z$, that is
\[z:=Dw, \text{ for }w=\sum_{n=N}^\infty 2^{-2n} u_n.\]
Then $\partial \TV(u)=\emptyset$ as above, but $\partial J_z(u) \neq \emptyset$ for $J_z=|D\cdot-z|(\Omega)$. Say that $N=1$ and in the construction of the example we choose
\[S=\rho_{1/8}\ast 1_{[1/4,3/4]},\]
where $\rho_{1/8}$ is a standard mollifier supported on $(-1/8,1/8)$. Then the function defined by 
\[g(t)=\begin{cases} 8t &\text{ if }t \in (0, 1/8) \\ +1 &\text{ if }t \in (1/8, 3/8) \\ -8t+4 &\text{ if }t \in (3/8, 5/8) \\ -1 &\text{ if }t \in (5/8, 7/8) \\ 8t-8 &\text{ if }t \in (7/8, 1) \end{cases}\]
satisfies $|g(t)| \ls 1$, $g(t)=\sign(S'(t))$ whenever $S'(t)\neq 0$, and $g' \in L^\infty(0,1)$. This implies that the negative divergence of the vector field $\tilde g$ defined by $\tilde g(x_1,x_2) = (g(x_1), 0)^T$ is in $\partial \TV(u_0)=\partial \TV(u-w) = \partial J_z(u)$, since $-\div \tilde g(x_1, x_2) = -g'(x_1)$ and
\[\int_\Omega -\div \tilde g(x) \,u(x) \dd x = \int_\Omega \tilde g(x) \cdot \nabla u(x) \dd x = \int_0^1 g(t)\, S'(t) \dd t = \int_0^1 |S'(t)| \dd t = \TV(u).\]
\end{remark}

\begin{remark}The $Du$ and $z$ of the example above are not in $\BD$. However, if one starts with a smoother $S \in W^{2,1}(0,1)$, the construction works just the same. We would get
\begin{gather*}\|\nabla^2 u_n\|_{L^1} = 2^{n} \|S''\|_{L^1(0,1)}\text{ but }\sum_{n=1}^\infty \, 2^{-2n} \cdot 2^{n} <+\infty, \\
\text{ so }u \in W^{2,1}(\Omega) \text{ and }z \in \big[W^{1,1}(\Omega)\big]^d\subset \BD(\Omega).\end{gather*}
\end{remark}

To summarize, the subdifferential of the functional $J_z=|D\cdot\,-z|(\Omega)$ being nonempty at $u$ does not automatically imply that $\partial \TV(u) \neq \emptyset$, which prevents us from easily applying the techniques of Section \ref{sec:boundedness}, in particular Remark \ref{rem:nonemptysubg}, to the $\TGV$ case. In any case, we note that our counterexamples are built with vector fields $z$ which are not optimal for the inner minimization of \eqref{eq:primal-tgv} defining $\TGV(u)$, so it could still be that $\partial \TGV(u) \neq \emptyset$ implies (beyond the trivial case in Proposition \ref{prop:uninterestingsubgrad}) that $\partial \TV(u) \neq \emptyset$.

\end{document}